\documentclass{amsart}
\usepackage{color}
\usepackage[pdfauthor={Dimitrios Chatzakos, Yiannis N. Petridis},
pdftitle={},
            pdfkeywords={Hyperbolic lattice points},
             pdfcreator={Pdflatex}]
{hyperref}
\hypersetup{colorlinks=true}
\usepackage{xcolor}
\usepackage[normalem]{ulem}
\usepackage{marginnote}
\usepackage{enumerate}

\usepackage{amssymb,amsmath,latexsym,amsthm}
\usepackage[english]{babel}

\newtheorem{theorem}{Theorem}[section]
\newtheorem{lemma}[theorem]{Lemma}
\newtheorem{proposition}[theorem]{Proposition}

\theoremstyle{definition}

\newtheorem{conjecture}[theorem]{Conjecture}

\numberwithin{equation}{section}
\DeclareMathOperator*{\vol}{vol}

\begin{document}

\newtheorem{remark}[theorem]{Remark}



\def \g {{\gamma}}
\def \G {{\Gamma}}
\def \l {{\lambda}}
\def \a {{\alpha}}
\def \b {{\beta}}
\def \f {{\phi}}
\def \r {{\rho}}
\def \R {{\mathbb R}}
\def \H {{\mathbb H}}
\def \N {{\mathbb N}}
\def \C {{\mathbb C}}
\def \Z {{\mathbb Z}}
\def \F {{\Phi}}
\def \Q {{\mathbb Q}}
\def \e {{\epsilon }}
\def \ev {{\vec\epsilon}}
\def \ov {{\vec{0}}}
\def \GinfmodG {{\Gamma_{\!\!\infty}\!\!\setminus\!\Gamma}}
\def \GmodH {{\Gamma\backslash\H}}
\def \sl  {{\hbox{SL}_2( {\mathbb R})} }
\def \psl  {{\hbox{PSL}_2( {\mathbb R})} }
\def \slz  {{\hbox{SL}_2( {\mathbb Z})} }
\def \pslz  {{\hbox{PSL}_2( {\mathbb Z})} }
\def \L  {{\hbox{L}^2}}

\newcommand{\norm}[1]{\left\lVert #1 \right\rVert}
\newcommand{\abs}[1]{\left\lvert #1 \right\rvert}
\newcommand{\modsym}[2]{\left \langle #1,#2 \right\rangle}
\newcommand{\inprod}[2]{\left \langle #1,#2 \right\rangle}
\newcommand{\Nz}[1]{\left\lVert #1 \right\rVert_z}
\newcommand{\tr}[1]{\operatorname{tr}\left( #1 \right)}

\title[Hyperbolic lattice-point counting]{ The hyperbolic lattice point problem\\ in conjugacy classes}
\author{Dimitrios Chatzakos}
\address{Department of Mathematics, University College London, Gower Street, London WC1E 6BT}
\email{d.chatzakos.12@ucl.ac.uk}
\author{Yiannis N. Petridis}
\address{Department of Mathematics, University College London, Gower Street, London WC1E 6BT}
\email{i.petridis@ucl.ac.uk}
\thanks{The first author was  supported by a DTA from  EPSRC during his PhD studies at UCL}
\date{\today}
\keywords{lattice points, hyperbolic space, geodesic arcs}
\subjclass[2010]{Primary 11F72; Secondary 37C35, 37D40}

\begin{abstract}
For $\Gamma$ a cocompact or cofinite Fuchsian group, we study the hyperbolic lattice point problem in conjugacy classes, which is a modification of the classical hyperbolic lattice point problem.  We use large sieve inequalities for the Riemann surfaces $\GmodH$ to obtain average results for the error term, which are conjecturally optimal. We give a new proof of  the error bound $O(X^{2/3})$, due to A. Good. For $\slz$ we interpret our results in terms of indefinite quadratic forms.
\end{abstract}
\maketitle
\section{Introduction}\label{Introduction}

Let $\mathbb{H}$ be the hyperbolic plane, $z$, $w$ be two fixed points in $\mathbb{H}$ and $\rho(z,w)$ denote the hyperbolic distance. Let also $\Gamma \subset \psl$ be a cocompact or cofinite Fuchsian group. The classical hyperbolic lattice point problem asks to estimate the number of points in the orbit $\Gamma z$ that belong in a disk of radius $R$ and center $w$, i.e. to give an asymptotic formula for 
\begin{displaymath}\# \{ \gamma \in \Gamma : \rho( \gamma z, w) \leq R \}.\end{displaymath}
Let $\cosh \rho(z,w) = 2 u(z,w) + 1$, where $u(z,w)$ is the standard point-pair invariant function
\begin{eqnarray*}
u(z,w) = \frac{|z-w|^2}{4 \Im(z) \Im(w)},
\end{eqnarray*} and define 
\begin{displaymath}H(X; z,w) = \# \{ \gamma \in \Gamma : 4 u (z,\gamma w) + 2 \leq X \}.\end{displaymath}
Selberg \cite{selberg}  proved that 
\begin{equation}\label{equat}H(X; z, w) = \sum_{1/2 < s_j \leq 1} \pi^{1/2} \frac{\Gamma(s_j - 1/2)}{\Gamma(s_j+1)} u_j(z) \overline{u_j(w)} X^{s_j} + E(X; z,w),\end{equation}
with 
\begin{displaymath}E(X;z,w) =  O(X^{2/3}),\end{displaymath}
where $\{u_j\}_{j=1}^{\infty}$ is an orthonormal system of eigenfunctions for the discrete spectrum of the hyperbolic Laplacian and the sum is over the small eigenvalues $\lambda_j = s_j (1 - s_j)<1/4$ of the surface $\Gamma \backslash \mathbb{H}$. For earlier results and extensions see \cite{patt, gunth}. 

 Now, let $\mathcal{H} \subset \Gamma$ be a hyperbolic conjugacy class of $\Gamma$. Write $\mathcal{H}$ as $\mathcal{H} = \mathcal{P}^{\nu}$ where $\mathcal{P}$ is a primitive conjugacy class, i.e. $\mathcal{H} = \{a g^{\nu} a^{-1}, a \in \Gamma \}$, where $g$ is a primitive hyperbolic element of $\Gamma$. For $\g \in \Gamma$ define
\begin{displaymath}\mu (\gamma) =\inf_{z \in \mathbb{H}} \rho(z,\gamma z).\end{displaymath}
 Notice that $\mu(\gamma)$ is constant in conjugacy classes, hence we can define $\mu := \mu(\mathcal{H}) = \mu(g^{\nu})$, which is the length of the closed geodesic corresponding to the hyperbolic class $\mathcal{H}$.

 Let $z$ be a fixed point in $\mathbb{H}$, and define the quantity
\begin{displaymath}N_z (t) = \# \{ \gamma \in \mathcal{H} : \rho(z,\gamma z) \leq t \}.
\end{displaymath}
For $\Gamma$ cocompact Huber \cite{hub} was the first one who posed and studied the problem of estimating the asymptotic behavior of $N_z (t)$,  as $t \to \infty$. 
 Huber proved that the asymptotic behavior of $N_z (t)$ as $t \to \infty$ is:
\begin{equation} \label{hubermainterm}
N_z (t) \sim \frac{2}{\vol(\GmodH)} \frac{\mu}{\nu} X,
\end{equation}
where
\begin{equation}\label{changeofvariable}X= \frac{\sinh(t/2)}{\sinh(\mu/2)},\end{equation}
and $\vol(\GmodH)$ is the area of $\Gamma \backslash \mathbb{H}$.

 There is a nice geometric interpretation of this problem, explained in \cite{hub} and \cite{huber}. Let $\ell$ be the invariant closed geodesic of $g$. Then, $N_z(t )$ counts the number of $\gamma \in \Gamma / \langle g \rangle $ such that $\rho(\gamma z, \ell) \leq t$. This is the number of geodesic segments on $\GmodH$ from $z$ perpendicular to $\ell$ of length less than or equal to $t$. After conjugation, one can assume that $\ell$ lie on $\{yi, y>0\}$.  Huber's interpretation shows that $N_z(t)$ actually counts $\gamma$ in $\Gamma/  \langle g \rangle $ such that $\cos v \geq X^{-1}$, where $v$ is the angle defined by the ray from $0$ to $\gamma z$ and the geodesic $\{yi, y>0\}$.

 For $\Gamma$ cocompact or cofinite, Good \cite{good} proved a general sum formula that covers many cases of decompositions of the group $G = \sl$. One of these cases corresponds to  Huber's hyperbolic lattice point problem in conjugacy classes. In Good's notation the hyperbolic lattice point problem in conjugacy classes corresponds to the $_{\eta} G_{\zeta}$ case, whereas the classical one corresponds to the $_{\zeta} G_{\zeta}$ case \cite[p.~20, Eq. (3.12)]{good}. His method is based on defining certain Poincar\'e series $P_{\xi} (z, s, m)$ \cite[p.~73, Eq. (7.1)]{good}  as sums over cosets of a hyperbolic subgroup of $\Gamma$  of his basic eigenfunctions $V_{\xi}(z, s, \lambda)$ \cite[p.~28, Eq.(4.8)]{good}. These Poincar\'e series generalise the Eisenstein series and the resolvent kernel. He then expands a modification of $P_{\xi}(z, s, m)$ into automorphic eigenfunctions, and computes the Fourier expansion around $\zeta$. This involves generalizations of Kloosterman sums, leading to a local trace formula \cite[p.~98, Theorem 1]{good}. The end result is  Good's general formula \cite[Theorem 4, p.~116]{good}.  After matching notation  for $m=n=0$  this formula implies
 \begin{eqnarray*}
N_z(t) = \frac{2}{\vol(\GmodH)} \frac{\mu}{\nu} X + 2 \lambda_{\mathcal{H}} \lambda_{z} \sum_{\frac{1}{2} < s_j < 1} a_j( \mathcal{H}, z) X^{s_j} + E_z(t), \end{eqnarray*}
where 
\begin{displaymath}
E_z(t) = O(X^{2/3}),
\end{displaymath}
$\lambda_{\mathcal{H}}$, $\lambda_{z}$ are specific constants and $a_j(\mathcal{H}, z)$ are functions depending on $s_j, u_j, \mathcal{H}, z$ and special functions.

Later Huber \cite{huber} proved that, for $\Gamma$ cocompact:
$$\Big| N_z(t) - \frac{2}{\vol(\GmodH)} \frac{\mu}{\nu} X \Big| \leq c_1 X^{\tau} + c_2 X^{3/4} + c_3 X^{1/2},$$
where $\tau$ is an parameter related to the eigenvalue $\lambda_1$ and $c_1, c_2, c_3$ are specific constants.

As the natural parametrization is given by  (\ref{changeofvariable}), we denote
$$N(\mathcal{H}, X; z)=N_z(t),$$
and work with $ N(\mathcal{H}, X; z)$ for the rest of this paper. Clearly
\begin{displaymath}N(\mathcal{H}, X; z) = \# \left\{\gamma \in \mathcal{H} : \frac{\sinh(\rho(z,\gamma z)  /2)}{\sinh(\mu/2)}\leq X \right\}.\end{displaymath}
 We also define 
 \begin{equation}\label{main-term}
M(\mathcal{H}, X; z) = \sum_{1/2 < s_j \leq 1} A(s_j)  \hat{u}_j u_j(z) X^{s_j},
\end{equation}
where
\begin{equation}\label{a-function} A(s) = 2^{s-1} \left( e^{i \frac{\pi}{2}(s-1)} + e^{-i \frac{\pi}{2}(s-1)} \right) \frac{\Gamma \left(\frac{s+1}{2} \right) \Gamma \left(1 - \frac{s}{2} \right) \Gamma \left(s-\frac{1}{2} \right)}{ {\pi}\Gamma(s+1)},
\end{equation}
\begin{equation}\label{period} \hat{u}_j = \int_{\sigma} u_j \,ds \end{equation}
is the period integral of $u_j$ across a segment $\sigma$ of the invariant geodesic of $g^{\nu}$ with length $\int_{\sigma} ds = \mu / \nu$ (see Lemma \ref{lemmahuber}) and the sum in \eqref{main-term} is over the small eigenvalues of the hyperbolic Laplacian of $\Gamma \backslash \mathbb{H}$. We denote by $E(\mathcal{H}, X; z)$ the error term
\begin{displaymath}
E(\mathcal{H}, X; z) = N(\mathcal{H}, X; z) - M(\mathcal{H}, X; z).
\end{displaymath}
In section \ref{hubertransform} we refine the machinery of Huber in \cite{huber}. We compute his special functions $\xi_{\lambda}(v)$ (see \cite[Eq.(10), (11)]{huber}) in terms of the Legendre functions $P^0_{s-1}(i\tan v)$.  This allows to show the oscillatory behaviour of the Huber transform $d(f^{\pm}, t)$, see proposition \ref{estimates1}. In sections \ref{cocompactcase} and \ref{cofinitecase} we give a new proof of the following theorem.
\begin{theorem}{\cite[Theorem 4, p.~116]{good}} \label{theorem1} Let $\Gamma$ be a cocompact or cofinite Fuchsian group, and $\mathcal{H}$ a hyperbolic conjugacy class of $\Gamma$. Then
$$E(\mathcal{H}, X; z) = O(X^{2/3}).$$
\end{theorem}
 We combine the results and techniques of section \ref{hubertransform} with the large sieve  inequalities obtained by Chamizo in \cite{cham1} to prove average results for the error term $E(\mathcal{H}, X; z)$, similar to those in \cite{cham2} for the error term of the classical hyperbolic lattice point problem. In section \ref{averagingresults} we prove the main results of this paper.
\begin{theorem}\label{average1}  Let $\Gamma$ be a cocompact or cofinite Fuchsian group, and $\mathcal{H}$ a hyperbolic conjugacy class of $\Gamma$. Then
$$\frac{1}{X} \int_{X}^{2X} |E(\mathcal{H}, x  ; z)|^2 dx  \ll  X \log^2 X,$$
where the constant implied in  {\lq $\ll$\rq}  depends on $\G, \mathcal{H}$ and $z$.
\end{theorem}
\begin{theorem} \label{average2} Let $\Gamma$ be a cocompact Fuchsian group, and $\mathcal{H}$ a hyperbolic conjugacy class of $\Gamma$. Then, for $n=1,2$
$$ \int_{\GmodH} |E(\mathcal{H}, X ; z)|^{2n} d\mu(z) \ll  X^{n} \log^{2n} X,$$
where the constant implied in  {\lq $\ll$\rq}  depends on $\G$ and  $\mathcal{H}$.
\end{theorem}
 Finally in section \ref{arithmeticapplications} we use Huber's geometric interpretation to study some arithmetic consequences of our results. More specifically, for $\mathcal{H}$ a hyperbolic class of $\slz$, we interpret the quantity $N(\mathcal{H}, X; z)$ in terms of  the number of solutions of  indefinite quadratic forms in four variables with restrictions.

 \begin{remark}
 Our use of piecewise linear functions $f^{\pm}$ to define the smooth automorphic functions
 $A(f^{\pm})(z)$ is much simpler than the construction and spectral expansion of Poincar\'e series in Good \cite{good}. Moreover,  the oscillatory behaviour that is crucial in the application of the large sieve seems difficult to identify in the local trace formula in \cite{good}. Even matching Good's expansion \cite[Theorem~4, p.~116]{good} with $M(\mathcal{H}, X; z)$ seems to be a complicated task needing extensive calculations. Only the leading term of  $M(\mathcal{H}, X; z)$ is easy to match.
 \end{remark}

\begin{remark} Eskin and McMullen used ergodic methods to study the asymptotics of various counting problems on Lie groups, one of which is the conjugacy class problem \cite[III.2, p.~187]{eskin}. Duke, Rudnick and Sarnak \cite{duke}  give another proof of the main term \cite[Example 1.5, p.~147]{duke}. For the classical hyperbolic lattice point problem one deals with the locally symmetric space $\G \backslash \sl / SO(2)$. Our case involves the space  $\G \backslash \sl / A $ which is not even Hausdorff, where $A$ is the group of diagonal matrices. \end{remark} 

\begin{remark}Hill and Parnovski  \cite{hillparn} studied the asymptotic behavior of the variance of the hyperbolic lattice point counting function for the classical hyperbolic lattice point  problem. When $\G$ has no small eigenvalues, in Theorem \ref{average2} we provide an upper bound for the variance of the hyperbolic lattice point  function in our situation.\end{remark}

\begin{remark} Recently Parkkonen and Paulin \cite{parkpaul} studied the hyperbolic lattice point problem in conjugacy classes for higher dimensional negatively curved manifolds. 
In special cases i.e. for compact manifolds or arithmetic group of isometries, they obtain bounds for the error term. It would be interesting to prove error bounds analogous to Theorem \ref{theorem1}, and average results analogous to Theorems \ref{average1}, \ref{average2}.
\end{remark}

\begin{remark} An interesting application of the hyperbolic lattice point problem in conjugacy classes and its geometric interpretation concerns degenerating Riemann surfaces and the appearance of Eisenstein series, see \cite{gjm}.
\end{remark}

\section{Certain automorphic functions and their spectral expansion}\label{hubertransform}

\subsection{ Plan of proof and comparison with the classical problem.}

Let $K(z,w)$ be the automorphic kernel defined as
\begin{displaymath}
K(z,w) = \sum_{\gamma \in \G} k(u(\gamma z, w)),
\end{displaymath}
for a test function $k(u)$. If $k(u)$ is the characteristic function of the interval $[0, {(X-2)}/{4}],$ then $K(z,w) =H(X;z,w)$, and the asymptotics of $H(X;z,w)$ can be studied using the pre-trace formula for the kernel $K(z,w)$. 
In practice one needs approximations $k_{\pm}(u)$ of $k(u)$ and estimates of the corresponding Selberg--Harish-Chandra transforms $h_{\pm}(t)$ of $k_{\pm}(u)$.

 When we restrict the summation to the conjugacy class $\mathcal{H} \subset \G$, we do not get an automorphic kernel $g(u)$ in the place of $k(u)$. Therefore, Selberg theory does not apply in this case. Huber, however,  defined an automorphic function $A(f)$ that plays the role of $K(z,w)$ for a suitable test function $f$ when $z=w$, see \eqref{af}. The spectral expansion of $A(f)$ provides the asymptotic behavior of $N(\mathcal H, X; z)$.
 
Assume that $\Gamma$ is cocompact. Let $C_0^*[1, \infty)$ be the space of real functions of compact support that are bounded in $[1, \infty)$ and have at most finitely many discontinuities. For an $f$ in $C_0^*[1, \infty)$, define the $\Gamma$-automorphic function 
\begin{equation}\label{af}
A(f) (z)= \sum_{\gamma \in \mathcal{H}} f \left( \frac{\cosh \rho(z,\gamma z) -1}{\cosh \mu(\gamma)-1} \right) .
\end{equation}
Since $\Gamma$ is cocompact and $f$ has compact support, the sum in \eqref{af} is finite. 

Let $\Delta$ be the hyperbolic Laplace operator on $\GmodH$ and $\{u_j \}_{j=0}^{\infty}$ be an orthonormal system of (real-valued) automorphic eigenfunctions of $-\Delta$, with corresponding eigenvalues $\{ \lambda_j\}_{j=0}^{\infty}$. Then $A(f)$ has an $L^2$-expansion:
$$A(f) = \sum_{j} c (f, t_j) u_j(z),$$
where $\lambda_j = 1/4+ t_j^2$ and 
$$c(f, t_j) = \int_{\GmodH} A(f)(z) u_j(z) \,d \mu(z)$$
is the $j$-th Fourier coefficient of $A(f)$. We have the following lemma:
\begin{lemma}[Huber, \cite{huber}]\label{lemmahuber} We have
$$c(f,t_j) = 2 \hat{u}_j d (f, t_j),$$
where $\hat{u}_j$ is the integral 
\begin{equation}\label{periodintegral}
\hat{u}_j=\int_{\sigma} u_j ds
\end{equation}
across a segment $\sigma$ of the invariant geodesic of $g$ with length $\int_{\sigma} ds = \mu / \nu$,
\begin{equation}\label{coefficients}
d(f, t) = \int_{0}^{\frac{\pi}{2}} f \left(\frac{1}{\cos^2 v} \right) \frac{\xi_{\lambda}(v)}{\cos^2 v} dv,
\end{equation}
with $\lambda=1/4+t^2$, and $\xi_{\lambda}$  is the solution of the differential equation
\begin{equation}\label{huberresult}\xi_{\lambda}''(v) + \frac{\lambda}{\cos^2 v} \xi_\lambda(v) = 0, \quad v \in \Big(-\frac{\pi}{2}, \frac{\pi}{2} \Big), \end{equation}
with $\xi_{\lambda}(0)=1$, $\xi_{\lambda}'(0)=0$.
\end{lemma}
 The coefficient $d(f,t)$, which we call the Huber transform of $f$, now plays the role of the Selberg--Harish-Chandra transform. For our choice of test functions $f$  we can use properties of special functions to estimate the Fourier coefficients  $d(f, t_j)$, see Proposition \ref{estimates1}. The next table summarizes the analogies  between the two problems.
 \\
\begin{center}
    \begin{tabular}{ | c | c | c | }
    \hline
    Classical problem & For conjugacy classes \\ \hline
          $z,w$ & $z, \mathcal{H}$ \\ \hline
          $k(u)$ & $f \left( \frac{\cosh \rho -1}{\cosh \mu-1} \right)$ \\ \hline
    $K(z,w)$ & $A(f)(z)$ \\ \hline
    $h(t)$ & $d(f,t)$ \\ \hline
    $u_j(w)$ & $\hat{u}_j$ \\ 
    \hline
    \end{tabular}
\end{center}

In order to bound  $E(\mathcal{H}, X; z)$, we will  need the following bound (weak type of Weyl's law) for the period integrals of $\hat{u}_j$'s defined in Lemma \ref{lemmahuber}: 
\begin{lemma}[Huber, {\cite[ eq.~(63), p.~24]{huber}} ]\label{lemmahuber2} For the sequence of the period integrals $\{ \hat{u}_j \}_{j=0}^{\infty}$, the following estimate holds:
$$ \sum_{t_j \leq T} |\hat{u}_j|^2 \ll T.$$
\end{lemma}
The exact asymptotic behavior was first proved by Good \cite[Theorem 2, p. 108]{good}, see also \cite{martin} and in bigger generality by Tsuzuki \cite[Theorem 1, p.~2]{tsuzuki}.

\subsection{Special functions and test functions}

For the proof of Theorem \ref{theorem1} it is crucial to identify the special function $\xi_\lambda(v)$ and its relevant properties. Using \cite[p.~185, eq.~(87)]{fay}, \cite[p.~111, eq.~(10), (12)]{bateman} and \cite[p.~1009, eq.~(9.132.2)]{gradry} we see that the general solution of equation $(\ref{huberresult})$ can be written in the form
\begin{displaymath}
\xi_{\lambda}(v) = a(s) F\left(s,1-s,1; \frac{1 - i \tan(v)}{2}\right) + b(s)  F\left(s,1-s,1; \frac{1 + i \tan(v)}{2} \right),
\end{displaymath}
where $F(a, b, c ; z)$ is the Gauss hypergeometric function. The initial conditions of Lemma \ref{lemmahuber} imply that 
$$a(s) = b(s) = \left( 2\cdot F(s,1-s,1; 1/2) \right)^{-1}.$$
Using \cite[p.~959, eq.~(8.702)]{gradry} and \cite[p.~104, eq.~(50)]{bateman} we can write $\xi_{\lambda}(v)$ as
$$\xi_{\lambda}(v)= (2 \sqrt{\pi})^{-1} \Gamma \left(\frac{s+1}{2} \right) \Gamma \left(1 - \frac{s}{2} \right) \left( P_{s-1}^0 \left(i \tan v \right) + P_{s-1}^0 \left(-i \tan v\right) \right),$$
where $P_{\nu}^{\mu}(z)$ is the associated Legendre function of the first kind. Using the change of variable $x=\tan(v)$, we get
$$d(f, t) =  (2 \sqrt{\pi})^{-1} \Gamma \left(\frac{s+1}{2} \right) \Gamma \left(1 - \frac{s}{2} \right) \int_{0}^{\infty} f(x^2+1) \left( P_{s-1}^0 (i x) + P_{s-1}^0 (-i x) \right) dx.$$
Huber's interpretation shows that we are counting $\g \in \mathcal{H}$ such that $(\cos v)^{-1} \leq X$, i.e. $x^2+1\leq X^2$. Hence, choosing 
$$f(x^2+1) = \left\{ \begin{array}{rcl}
1, & \mbox{for} & x \leq \sqrt{X^2-1}, 
\\ 0, & \mbox{for} & x > \sqrt{X^2-1},
\end{array} \right. $$
we get
$$A(f)(z) = N(\mathcal{H}, X ;z).$$
Let us set 
\begin{equation}\label{uversusx}
U = \sqrt{X^2 - 1}.
\end{equation}
Motivated by \cite[p.~269]{bruin} we  define the following test functions for $x>0$ and $0 < U/2< T < U < V < 2U$:
\begin{equation}\label{fplus}
f^+(x^2 +1) = \left\{ \begin{array}{lcl}
1, & \mbox{for} & x \leq U, 
\\ \displaystyle \frac{V-x}{V-U}, & \mbox{for} & U \leq x \leq V,
\\ 0, & \mbox{for} & V \leq x,
\end{array} \right. 
\end{equation}
\begin{equation}\label{fminus}
f^-(x^2 + 1) = \left\{ \begin{array}{lcl}
1, & \mbox{for} & x \leq T,
\\ \displaystyle\frac{U-x}{U-T}, & \mbox{for} & T \leq x \leq U,
\\ 0, & \mbox{for} & U \leq x.
\end{array} \right. 
\end{equation}
Denote $Y = V-U$. Notice that 
$$f(x^2 + 1) = \left\{ \begin{array}{rcl}
1, & \mbox{for} & x \leq U,
\\ 0, & \mbox{for} & U < x,
\end{array} \right. $$
hence $f^- \leq f \leq f^+$. This gives 
$$A(f^-)(z) \leq N(\mathcal{H}, X ; z) \leq A(f^+)(z).$$
Since $U=X+O(X^{-1})$ as $U, X \to \infty$, we can translate  estimates involving $X$ to ones with  $U$ and vice versa. 
 We compute $d(f^+, t)$ and $d(f^-, t)$. The analysis for $d(f^-, t) $  is similar to the one for $d(f^+, t)$ with  $U$ and $T$ instead of $V$ and $U$. Therefore, we consider only $d(f^+, t)$.

For an $A>0$, we define $I(A)$ and $J(A)$ by
$$I(A) = \int_{0}^{A} \left( P_{s-1}^0 (i x) + P_{s-1}^0 (-i x) \right) (A-x) dx,$$
$$J(A) = (A^2 +1) \left( P_{s-1}^{-2}(iA) + P_{s-1}^{-2}(-iA) \right).$$
Then, it is easy to see that
$$d(f^+, t) =   (2 \sqrt{\pi})^{-1} \Gamma \left(\frac{s+1}{2} \right) \Gamma \left(1 - \frac{s}{2} \right) \cdot \frac{ I(V) - I(U)}{V-U}.$$
\begin{lemma}\label{integrals}The functions $I(A)$ and $J(A)$ satisfy the relation
$$I(A) = J(A) -2P_{s-1}^{-2}(0).$$
\end{lemma}
\begin{proof}
Using integration by parts, the formula \cite[p.~968, eq.~8.752.3]{gradry}, and the fact that the function $(z^2-1)^{1/2} P_{s-1}^{-1}(z)$ is single-valued in the disk with center $(1,0)$ and radius $2$, we get
$$I(A) =  -i \int_{0}^{A} (-x^2-1)^{1/2} \left( P_{s-1}^{-1}(ix) -  P_{s-1}^{-1}(-ix) \right) dx.$$
Using again twice \cite[p.~968, eq.~8.752.3]{gradry} for $m=1, 2$ we get 
$$I(A) = \left. (x^2+1) \Big( P_{s-1}^{-2}(ix) +  P_{s-1}^{-2}(-ix) \Big) \right|_0^A.$$
The result is immediate.
\end{proof}

\subsection{Estimates for the Huber transform}

Lemma \ref{integrals} implies that
\begin{equation} \label{coeff1}
d(f^+, t) = (2 \sqrt{\pi})^{-1} \Gamma \left(\frac{s+1}{2} \right) \Gamma \left(1 - \frac{s}{2} \right) \cdot \frac{ J(V) - J(U)}{V-U}.
\end{equation}
Relation \cite[p.~971, eq.~8.776.1]{gradry} implies 
\begin{eqnarray}\label{dft}
d(f^+, t) &=& B(s) \cdot \frac{ (V^2+1)V^{s-1} - (U^2 + 1) U^{s-1}}{V-U} \cdot \left(1+ O(U^{-2}) \right) \\ \nonumber
&+& D(s)  \cdot \frac{ (V^2+1)V^{-s} - (U^2 + 1) U^{-s}}{V-U} \cdot \left(1+ O(U^{-2}) \right), \\ \nonumber
\end{eqnarray}
where
\begin{equation} \label{betacoeff}
B(s) =  2^{s-2} \left( e^{i \frac{\pi}{2}(s-1)} + e^{-i \frac{\pi}{2}(s-1)} \right) \frac{\Gamma \left(\frac{s+1}{2} \right) \Gamma \left(1 - \frac{s}{2} \right) \Gamma \left(s-\frac{1}{2} \right) }{ {{\pi}}\Gamma(s+2)},
\end{equation}
\begin{equation} \label{deltacoeff}
D(s) = \left( e^{i \frac{\pi}{2}(-s)} + e^{-i \frac{\pi}{2}(-s)} \right) \frac{\Gamma \left(\frac{s+1}{2} \right) \Gamma \left(1 - \frac{s}{2} \right) \Gamma \left(\frac{1}{2} -s \right)}{{{\pi}}\Gamma(3-s) 2^{s+1}}.
\end{equation}

\begin{proposition}\label{estimates1} a) For any $s=1/2+it$ we have 
\begin{eqnarray*}
d(f^+,t) &=& B \left(\frac{1}{2}+it\right)  \left(\frac{3}{2}+it \right) X^{1/2+it} + D \left(\frac{1}{2}+it \right) \left(\frac{3}{2}-it \right)  X^{1/2-it} \\
&+& O \left( B \left(\frac{1}{2}+it\right) |t|^{2} X^{-1/2+it} Y + D \left(\frac{1}{2}+it \right) |t|^{2}  X^{-1/2-it} Y \right)
\end{eqnarray*}
b) Let $t \in \mathbb{R}$ (i.e $\Re(s) = 1/2$) and $t \neq 0$.  Then, $d(f^+,t)$ can be written in the form 
$$d(f^+, t) = a(t, Y/X) X^{1/2+it} + b(t, Y/X) X^{1/2-it},$$
where the coefficients $a(t,Y/X)$ and $b(t, Y/X)$ satisfy the bound
$$a(t,Y/X), b(t, Y/X) =O \left(t^{-2} \min \{t, X Y^{-1} \}\right).$$
Hence
$$d(f^+, t) = O \left(t^{-2} \min \{t, X Y^{-1} \} X^{1/2}\right).$$
c) Let $t \notin \mathbb{R}$, i.e $s \in (1/2, 1]$. Then
\begin{eqnarray*}
d(f^+,t) &=& B(s) (s+1) X^s + D(s) (2-s) X^{1-s}\\
 &&+ O(\G(s-1/2)Y+ \G(1/2-s) X^{1/2}).
\end{eqnarray*}
d) For $t=0$ we get
$$d(f^+, 0) = O(X^{1/2} \log X).$$
\end{proposition}
\begin{proof}
a) First, apply the mean value theorem to the function $f(x) = x^{s+1} + x^{s-1}$ to get
$$ \frac{ (V^2+1)V^{s-1} - (U^2 + 1) U^{s-1}}{V-U} = (s+1) X^s + O( s (s+1) X^{s-1} Y + (s-2) X^{-1}).$$
Applying it again to the function  $g(x) = x^{2-s} +  x^{-s}$, we have
$$ \frac{ (V^2+1)V^{-s} - (U^2 + 1) U^{-s}}{V-U} = (2-s) X^{1-s} + O((2-s) (1-s) X^{-s} Y + (-s) X^{-3/2}).$$
Plugging $s=1/2+it$ in (\ref{dft}) and using that $O(U^{-2}) = O(X^{-2})$ and the above estimates, we get the result.
\\
b) First, consider the function $f(x)$ as above. We know from part $a)$ that the terms containing $X^{1/2+it}$ come from the terms contaning $f(x)$. The mean value theorem imples
$$ \frac{ (V^2+1)V^{s-1} - (U^2 + 1) U^{s-1}}{V-U} \ll |t| \cdot X^{1/2},$$
whereas, trivial estimates imply
$$ \frac{ (V^2+1)V^{s-1} - (U^2 + 1) U^{s-1}}{V-U} \ll X^{3/2} Y^{-1}.$$
Hence, if we set
$$a(t, Y/X) = \left(1+ O(U^{-2}) \right)  B(s) \cdot \frac{ (V^2+1)V^{s-1} - (U^2 + 1) U^{s-1}}{V-U}  X^{-(1/2+it)},$$
and use the Stirling's formula for the $\Gamma$ function, we get the bound
$$a(t,Y/X) =O \left(t^{-2} \min \{t, X Y^{-1} \}\right).$$
Doing the same for $g(x)$ as above and the coefficient $b(t, Y/X)$ defined as
$$b(t, Y/X) = \left(1+ O(U^{-2}) \right)  D(s) \cdot \frac{ (V^2+1)V^{-s} - (U^2 + 1) U^{-s}}{V-U}  X^{-(1/2-it)},$$
we get $b)$.
\\
c) It follows from $a)$. We estimate three of the $\G$-factors in $B(s)$, $D(s)$ (eq. (\ref{betacoeff}), (\ref{deltacoeff})) and keep the factors $\Gamma (s-1/2)$ and  $\Gamma (1/2-s)$ accordingly.
\\
d) Putting $t=0$ in (\ref{coeff1}) we get
\begin{eqnarray*}
d(f^+,0) = (2 \sqrt{\pi} )^{-1} \Gamma^2 (3/4) \frac{ H(V) - H(U) }{V-U}
\end{eqnarray*}
where $H(z) =  (z^2 +1) \left( P_{-1/2}^{-2}(iz) + P_{-1/2}^{-2}(-iz) \right)$. Thus, applying once again the mean value theorem, there exists a $\xi \in [U, V]$ such that
\begin{eqnarray*}
d(f^+,0) = (2 \sqrt{\pi} )^{-1} \Gamma^2 (3/4) H'(\xi).
\end{eqnarray*}
For $H'(z)$ we have
$$
H'(z) =  2z \left( P_{-1/2}^{-2}(iz) + P_{-1/2}^{-2}(-iz) \right) + (z^2 + 1) \frac{d}{dz} \left(  P_{-1/2}^{-2}(iz) + P_{-1/2}^{-2}(-iz) \right).
$$
Formula \cite[p.~964, eq.~(8.731.1)]{gradry} implies
\begin{equation} \label{hprime}
H'(z) =  \frac{3z}{2} \left( P_{-1/2}^{-2}(iz) - P_{-1/2}^{-2}(-iz) \right) -\frac{5i}{2} \left(  P_{1/2}^{-2}(iz) - P_{1/2}^{-2}(-iz) \right).
\end{equation}
Consider the first bracket. Using formula \cite[p.~961, eq.~(8.713.2)]{gradry} we get
\begin{eqnarray*}
P_{-1/2}^{-2}(i\xi) - P_{-1/2}^{-2}(-i\xi) &\ll& (\xi^2+1) \int_{0}^{\infty} \left(\cosh^2 t + \xi^2 \right)^{-5/4} dt \\
&\ll& \xi^{-1/2} \int_{0}^{\infty} \left( \left(\frac{\cosh t}{\xi}\right)^2 + 1 \right)^{-5/4} dt.
\end{eqnarray*}
Setting $x = \cosh t / \xi$ we get
\begin{eqnarray*}
\int_{0}^{\infty} \left( \left(\frac{\cosh t}{\xi}\right)^2 + 1 \right)^{-5/4} dt &=&  \int_{1/\xi}^{\infty} \left(x^2 + 1 \right)^{-5/4} 
\frac{\xi}{(\xi^2x^2 -1 )^{1/2}} dx \\
&=& \int_{1/\xi}^{1} \left(x^2 + 1 \right)^{-5/4} \frac{\xi}{(\xi^2x^2-1)^{1/2}} dx \\
&&+ \int_{1}^{\infty} \left(x^2 + 1 \right)^{-5/4} 
\frac{\xi}{(\xi^2x^2 -1)^{1/2}} dx.
\end{eqnarray*}
Since $U, V \to \infty$, we can assume that $\xi \geq 2$. We see that
$$ \int_{1}^{\infty} \left(x^2 + 1 \right)^{-5/4} \frac{\xi}{(\xi^2x^2-1)^{1/2}} dx \ll  \int_{1}^{\infty} \left(x^2 + 1 \right)^{-5/4} dx = O(1)$$
and, after setting $u=x \xi$,
\begin{eqnarray*}
 \int_{1/\xi}^{1} \left(x^2 + 1 \right)^{-5/4} \frac{\xi}{(\xi^2x^2-1)^{1/2}} dx &=& \int_{1}^{\xi} \left(\frac{\xi^2}{u^2+\xi^2}\right)^{5/4} \frac{\xi}{(u^2-1)^{1/2}} \frac{du}{\xi} \\
&\leq&  \int_{1}^{\xi} \frac{1}{\sqrt{u^2-1}} du \ll \log \xi.
\end{eqnarray*}
Combining these estimates we get
$$P_{-1/2}^{-2}(i\xi) + P_{-1/2}^{-2}(-i\xi) \ll \xi^{-1/2} \log \xi.$$
For the second bracket, using once again  \cite[p.~961, eq.~(8.713.2)]{gradry}, we get
\begin{eqnarray*}
 P_{1/2}^{-2}(i\xi) - P_{1/2}^{-2}(-i\xi) &\ll& (\xi^2+1) \int_{0}^{\infty} \cosh t \left(\cosh^2 t + \xi^2 \right)^{-5/4} dt \\
&\ll& \xi^{1/2} \int_{0}^{\infty} \frac{\cosh t}{\xi}  \left( \left(\frac{\cosh t}{\xi}\right)^2 + 1 \right)^{-5/4} dt.
\end{eqnarray*}
As above, set $x = \cosh t / \xi$ and split the integral into two integrals:
\begin{eqnarray*}
\int_{1/\xi}^{1} \left(x^2 + 1 \right)^{-5/4} \frac{\xi x}{(\xi^2x^2-1)^{1/2}} dx
+ \int_{1}^{\infty} \left(x^2 + 1 \right)^{-5/4} 
\frac{\xi x}{(\xi^2x^2 -1)^{1/2}} dx.
\end{eqnarray*}
As above, assuming $\xi \geq 2$, the second integral is easily seen to converge, whereas the first one, setting $u =x \xi$ is again bound by  $\int_{1}^{\xi} (u^2-1)^{-1/2} du$. Finally, combining all the above estimates, we get
$$d(f^+,0) \ll H'(\xi) \ll \xi^{1/2} \log \xi \ll V^{1/2} \log V,$$
which implies the desired bound, since $X \sim U$ and $V<2U$.
\end{proof}

\section{The cocompact case}\label{cocompactcase}
We can now prove Theorem \ref{theorem1} when $\Gamma$ is cocompact:
\begin{theorem} \label{theorempart1}Let $\Gamma$ be a cocompact Fuchsian group, and $\mathcal{H}$ a hyperbolic conjugacy class of $\Gamma$. Then the error term $E(\mathcal{H}, X;z)$ satisfies the bound
$$E(\mathcal{H}, X; z) = O(X^{2/3}).$$
\end{theorem}
\begin{proof}
We begin with the spectral expansion of $A(f^+)$:
\begin{eqnarray*}
A(f^+) (z) &=&  \sum_{j} c(f^+, t_j) u_j(z) =  \sum_{j} 2  d(f^+, t_j) \hat{u}_j u_j(z) 
\end{eqnarray*}
Using  Proposition \ref{estimates1}, we write it in the form
\begin{eqnarray*}
A(f^+) (z) &=&  \sum_{1/2 < s_j \leq 1}  2 B(s_j) (s_j +1) \hat{u}_j u_j(z) X^{s_j} + \sum_{1/2 < s_j \leq 1}  2D(s_j) (2-s_j) X^{1-s_j}\\
 &&+ O \left( \sum_{1/2 < s_j \leq 1} \G(s_j-1/2) \hat{u}_j u_j(z)Y+ \sum_{1/2 < s_j \leq 1}\G(1/2-s_j) \hat{u}_j u_j(z) X^{1/2} \right) \\
&&+ \sum_{0 \neq t_j \in \mathbb{R}} 2 d(f^+, t_j)  \hat{u}_j u_j(z)  +  O(X^{1/2} \log X).
\end{eqnarray*}
Since the spectrum is discrete, for $s_j$ corresponding to a small eigenvalue, $s_j- 1/2$ is bounded away from zero. As the number of small eigenvalues is finite, we get
\begin{eqnarray*}
 \sum_{1/2 < s_j \leq 1} \G(s_j-1/2) \hat{u}_j u_j(z)Y+ \sum_{1/2 < s_j \leq 1}\G(1/2-s_j) \hat{u}_j u_j(z) X^{1/2} = O (Y+X^{1/2}).
\end{eqnarray*}
By the same argument,
\begin{eqnarray*}
\sum_{1/2 < s_j \leq 1}  2D(s_j) (2-s_j) X^{1-s_j} = O(X^{1/2}).
\end{eqnarray*}
Let $A(s)$ be the function defined in Eq. (\ref{a-function}). Then 
\begin{eqnarray*}
A(s) = 2 B(s) (s+1),
\end{eqnarray*}
and after setting 
\begin{eqnarray*}
G(f^+,z) = \sum_{0 \neq t_j \in \mathbb{R}} 2 d(f^+, t_j)  \hat{u}_j u_j(z),
\end{eqnarray*}
we can rewrite the spectral expansion of $A(f^+) (z)$ as
\begin{equation} \label{afestimate}
A(f^+) (z) =  \sum_{1/2 < s_j \leq 1}  A(s_j) \hat{u}_j u_j(z) X^{s_j} + G(f^+,z)  +  O(Y + X^{1/2} \log X).
\end{equation}
Using again Proposition \ref{estimates1} and the discreteness of the spectrum, we get
\begin{eqnarray*}
G(f^+,z) &=& \sum_{|t_j|\geq 1} 2 d(f^+, t_j)  \hat{u}_j u_j(z) + \sum_{|t_j|<1} 2 d(f^+, t_j)  \hat{u}_j u_j(z)  \\
&=& \sum_{|t_j|\geq 1} 2 d(f^+, t_j)  \hat{u}_j u_j(z) +O(X^{1/2}).
\end{eqnarray*}
Since $d(f,t)$ is an even function of $t$, after using dyadic decomposition we get the bound 
\begin{eqnarray*}
\sum_{|t_j| \geq 1}  d(f^+, t_j)  \hat{u}_j u_j(z) &\ll& \sum_{t_j\geq 1}  d(f^+, t_j)  \hat{u}_j u_j(z) \\
&=& \sum_{n=0}^{\infty} \left( \sum_{2^n \leq t_j <2^{n+1}}  d(f^+, t_j)  \hat{u}_j u_j(z) \right) \\
&\ll& \sum_{n=0}^{\infty} \sup_{ 2^n \leq t_j < 2^{n+1}} d(f^+, t_j) \left( \sum_{2^n \leq t_j <2^{n+1}}  \hat{u}_j u_j(z) \right).
\end{eqnarray*}
Using \cite[prop.~7.2]{iwaniec}, Proposition \ref{estimates1} and Lemma \ref{lemmahuber2}, we get
\begin{eqnarray*}
G(f^+,z) &\ll& \sum_{n=0}^{\infty} 2^{-2n} \min \left\{2^n, X Y^{-1} \right\} X^{1/2} \left( \sum_{ t_j <2^{n+1}}  |\hat{u}_j|^2 \right)^{1/2} \left( \sum_{t_j < 2^{n+1}} |u_j(z)|^2 \right)^{1/2}  + X^{1/2} \\
&\ll& X^{1/2}  \left( \sum_{n=0}^{\infty}  2^{-n/2} \min \left\{2^n, X Y^{-1} \right\} \right) + X^{1/2}.
\end{eqnarray*}
We split the sum according to $n < \log_2(X/Y)$ and $n > \log_2(X/Y)$. We get
\begin{eqnarray*}
G(f^+,z) &\ll& X^{1/2}  \left( \sum_{n <\log_2(X/Y)}  2^{-n/2} \min \left\{2^n, X Y^{-1} \right\} \right) \\
&&+  X^{1/2} \left( \sum_{n \geq \log_2(X/Y)} 2^{-n/2}  \min \left\{2^n, X Y^{-1} \right\} \right) + X^{1/2},
\end{eqnarray*}
which is bounded by
\begin{eqnarray*}
X^{1/2}\sum_{n <\log_2(X/Y)}  2^{n/2} &+&  X^{3/2} Y^{-1}  \sum_{n \geq \log_2(X/Y)} 2^{-n/2}+  X^{1/2} \\
&\ll& X Y^{-1/2} + X^{1/2}.
\end{eqnarray*}
By (\ref{afestimate}) we finally get
\begin{equation}\label{efestimate} A(f^+)(z) = \sum_{1/2 < s_j \leq 1} A(s_j) \hat{u}_j u_j(z) X^{s_j} + O(XY^{-1/2} + Y + X^{1/2} \log X).
\end{equation}
We work similarly for $A(f^-)$ and we use
\begin{eqnarray*}
A(f^-) (z) \leq N(\mathcal{H}, X; z) \leq A(f^+)  (z)
\end{eqnarray*}
to obtain
\begin{eqnarray*}
E(\mathcal{H}, X;z) = O(XY^{-1/2} +Y + X^{1/2} \log X).
\end{eqnarray*}
The optimal error arises for $Y = XY^{-1/2}$, i.e. $Y = X^{2/3}$, which yields 
\begin{eqnarray*}
E(\mathcal{H}, X;z) = O(X^{2/3}).
\end{eqnarray*}
\end{proof}
\begin{remark} For $\lambda_0 =0$, i.e. $s_0 = 1$, the contribution to $M(\mathcal{H}, X;z)$ is $2 \hat{u}_0 u_0(z) X$, with
\begin{eqnarray*}
u_0 (z) =  \frac{1}{\sqrt{\vol(\GmodH)}}.
\end{eqnarray*}
It is immediate to see that
\begin{eqnarray*}
\hat{u}_0 =  \frac{1}{\sqrt{\vol(\GmodH)}} \frac{\mu}{\nu},
\end{eqnarray*}
hence we get Huber's main term (\ref{hubermainterm}).
\end{remark}

\section{The cofinite case}\label{cofinitecase}

Now, let $\Gamma$ be a cofinite Fuchsian group, and define $A(f)$ as in Eq. \ref{af}. The first obstacle we face is to examine whether $A(f)$ is in $L^2(\Gamma \backslash \mathbb{H})$. To see this, suppose that $f$ is compactly supported in $[1,K]$ with $K >0$ fixed, and consider the counting function
$$\tilde{N}(z,\delta) =  \# \{ \gamma \in \mathcal{H} : u(\gamma z,z) \leq \delta \}.$$
An element $\g$ contributes to the summation in $A(f)$ exactly when $\g \in \tilde{N}(z,\delta)$, with $\delta= K (\cosh(\mu) -1)$. To prove that $A(f)$ is in $L^2(\GmodH)$, it suffices to prove that $\tilde{N}(z,\delta)$ in uniformly bounded (i.e. independently of $z$).

\begin{lemma}\label{afinl2}The $\tilde{N}(z,\delta)$ in uniformly bounded, hence $A(f) \in L^2(\GmodH)$.
\end{lemma}
\begin{proof}
For simplicity, assume that $\g$ has only one cusp at $\frak{a}$. Conjugating, we can assume that $\frak{a} =\infty$. Then, for $Y>0$, consider the set 
\begin{eqnarray*}
A(Y) = \{ \g \in \G_{\infty} \backslash \G : \Im (\g z) >Y \}.
\end{eqnarray*}
Lemma \cite[Lemma 2.10, p.~50]{iwaniec} shows that
\begin{eqnarray*}
\# A(Y) <1+ \frac{10}{c_{\infty} Y},
\end{eqnarray*}
where $c_{\infty}$ is a constant depending only on the cusp $\infty$. That means there exists a $Y_0$ such that for every $Y>Y_0$, 
$$\# A(Y) \leq 1,$$
 i.e. for $Y$ large enough $A(Y)$ contains at most one class $\G_{\infty} \g$. Since the fundamental domain contains points of deformation $\leq 1$, we have $\Im(z) \geq \Im(\g z) >Y$. Since the trivial class leaves the $ \Im (\g z) = \Im(z)$, we have $\g \in  \G_{\infty}$. Hence, if $\g \in \mathcal{H}$, then, for all $z$, $\Im (\g z) <Y_0$, where $Y_0$ depends only on the cusp $\infty$.

 Now, let $\g$ be an element of $\mathcal{H}$ such that $u(\g z, z) \leq \delta$. Using the formula
\begin{eqnarray*}
u(z, w) = \frac{|z-w|^2}{ 4 \Im(z) \Im(w)}
\end{eqnarray*}
we obtain 
\begin{eqnarray*}
\Im(z) - \Im(\g z) \leq |z- \g z| \leq \sqrt{ 4 \delta \Im(z) \Im(\g z)} \leq \sqrt{ 4 \delta Y_0 \Im(z)},
\end{eqnarray*}
hence
\begin{eqnarray*}
\Im(z) \leq \sqrt{ 4 \delta Y_0 \Im(z)} + Y_0.
\end{eqnarray*}
This inequality implies an upper bound $\Im(z) \leq M$, where $M$ depends only on $Y_0$ and $\delta$. We also get
\begin{eqnarray*}
\Re(\g z) - \Re(z) \leq |z- \g z| \leq \sqrt{ 4 \delta \Im(z) \Im(\g z)} \leq \sqrt{ 4 \delta Y_0 M},
\end{eqnarray*}
and since, for $\G$ cofinite, we have a uniform bound $|\Re(z)| <M$, we get also get a uniform bound for $\Re(\g z)$. That means, for all $\g \in \mathcal{H}$ satisfying $u(\g z, z) \leq \delta$, $\g z$ lies in a compact set, which does not depend of $z$ but only on $\delta$. This proves  $\tilde{N}(z,\delta)$ in uniformly bounded, and thus $A(f) \in L^2(\GmodH)$.
\end{proof}

Lemma \ref{afinl2} allows us to write a spectral expansion for $A(f)$. For $\Gamma$ cofinite but not cocompact the continuous spectrum of $-\Delta$ covers the segment $[1/4, \infty)$ uniformly with multiplicity the number of cusps; the eigenfunction that corresponds to the eigenvalue $\lambda = {1}/{4} + t^2 \geq {1}/{4}$ is the Eisenstein series $E_{\mathfrak{a}} (z, 1/2+ it)$. The spectral expansion of $A(f)$ becomes
\begin{equation} \label{cofinitecaseexpansion}
A(f) = \sum_{j} c(f, t_j) u_j(z) + \sum_{\mathfrak{a}} \frac{1}{4 \pi}  \int_{-\infty}^{\infty} c_{\mathfrak{a}}(f,t) E_{\mathfrak{a}} (z, 1/2 + it ) dt.
\end{equation}
The rest of the proof for the cofinite case is the same as for the cocompact case: the estimates of $d(f,t)$ that we wrote in section \ref{cocompactcase} (Proposition \ref{estimates1}) do not depend on the compactness of the group $\Gamma$, but only on the spectral parameter $t$. To complete the proof of Theorem \ref{theorem1}, we need the analogues of Lemmas \ref{lemmahuber} and \ref{lemmahuber2} for Eisenstein series.

Examining the proof of Lemma \ref{lemmahuber} in \cite{huber}, we notice that, along the same lines, we can prove the following version for Eisenstein series.
\begin{lemma}\label{eisensteinlemma1} We have
$$c_{\mathfrak{a}}(f,t) = 2 \hat{E}_{\mathfrak{a}} (1/2+ it) d (f, t),$$
where $\hat{E}_{\mathfrak{a}} (1/2+ it)$ is the integral 
$$\hat{E}_{\mathfrak{a}} (1/2+ it)=\int_{\sigma} E_{\mathfrak{a}} (z, 1/2+ it) ds$$
across a segment $\sigma$ of the invariant geodesic of $\gamma$ with length $\int_{\sigma} ds = \mu / \nu$, $d(f, t)$ given by eq. (\ref{coefficients}) and $\xi_{\lambda}$ satisfying eq. (\ref{huberresult}) with the same initial conditions.
\end{lemma}
The analogue of Lemma \ref{lemmahuber2} for Eisenstein series is the following lemma.
\begin{lemma}\label{eisensteinlemma2} We have the bound
$$\int_{-T}^{T} \Big|\hat{E}_{\mathfrak{a}} ( 1/2+ it) \Big|^2 dt \ll T.$$
\end{lemma}
\begin{proof} For $T>0$, define the angle $v_T \in (0,\frac{\pi}{2})$ by the relation
$$\tan (v_T) = \frac{\sqrt{2}}{T},$$
and the function $f$ as
$$f(u)= \left\{ \begin{array}{rcl}
1, & \mbox{for} & 1\leq u \leq \cos^{-2}(v_T),
\\ 0, & \mbox{for} & \cos^{-2}(v_T) < u.
\end{array} \right. $$
Thus, for 
$$X= \sqrt{1+ \frac{2}{T^2}},$$
we have $A(f)(z) = N(\mathcal{H}, X;z)$. By lemma \ref{afinl2}, we get that $A(f)$ is in $L^2(\Gamma \backslash \mathbb{H})$. Moreover, lemma \ref{afinl2} shows that
\begin{eqnarray*}
M_X := \sup_{z \in \mathbb{H}} N(\mathcal{H}, X;z) <\infty.
\end{eqnarray*}
Since we are interested about the estimate as $T \to \infty$, $X$ remains bounded and hence $M_X$ can be chosen uniformly bounded by some $M$. Then, we have the trivial bound 
$$\int_{\GmodH} \left(A(f)(z)\right)^2 d\mu(z) \leq M \int_{\GmodH} A(f)(z) d\mu(z),$$
and, by \cite[p.~24, eq.~(60)]{huber}, we get the bound
$$ \int_{\GmodH} A(f)(z) d\mu(z) \ll T^{-1}.$$
By lemma \ref{eisensteinlemma1}, for $\lambda=1/4+t^2$ we get
$$c_{\mathfrak{a}}(f,t) = 2 \hat{E}_{\mathfrak{a}} ( 1/2+ it)  \int_{0}^{v_T} \frac{\xi_{\lambda}(v)}{\cos^2(v)} dv.$$
On the other hand, from the Parseval's identity we get
\begin{eqnarray*}
\int_{F} \left(A(f)(z)\right)^2 d\mu(z) &=& \sum_{j} |c(f,t_j)|^2 + \sum_{\mathfrak{a}} \frac{1}{4\pi} \int_{-\infty}^{+\infty} |c_{\mathfrak{a}}(f,t)|^2 dt\\
&\geq& \sum_{\mathfrak{a}} \frac{1}{4\pi} \int_{-T}^{T} |c_{\mathfrak{a}}(f,t)|^2 dt \\
&\geq& \sum_{\mathfrak{a}}\frac{1}{\pi}\int_{-T}^{T}\left|\hat{E}_{\mathfrak{a}}( 1/2+it)\right|^2 \left(\int_{0}^{v_T}\frac{\xi_{\lambda}(v)}{\cos^2(v)} dv\right)^2 dt.
\end{eqnarray*}
We use \cite[Appendix, eq.~(5), p.~39]{huber}, as in the proof of Lemma \ref{lemmahuber2} in \cite[p.~24]{huber} to get
$$\int_{0}^{v_T}\frac{\xi_{\lambda}(v)}{\cos^2(v)} dv \gg  T^{-1},$$
hence
$$\int_{-T}^{T}\left|\hat{E}_{\mathfrak{a}}(1/2+it)\right|^2 \ll T.$$
\end{proof}
We can now finish the proof of Theorem \ref{theorem1} as follows.
\begin{theorem} \label{theorempart2} Let $\Gamma$ be a cofinite Fuchsian group, and $\mathcal{H}$ a hyperbolic conjugacy class of $\Gamma$. Then
$$E(\mathcal{H}, X; z) = O(X^{2/3}).$$
\end{theorem}
\begin{proof}
The part that corresponds to the Maa{\ss} cusp forms can be handled exactly as in the cocompact case. For the contribution of Eisenstein series in the spectral expansion of $A(f)(z)$ we need the estimate
$$\sum_{\mathfrak{a}}  \int_{-\infty}^{\infty} d(f,t) \hat{E}_{\mathfrak{a}}(1/2+it) E_{\mathfrak{a}} (z, 1/2 + it ) dt = O(XY^{-1/2} + X^{1/2}).$$
We use Proposition \ref{estimates1}, Lemmas \ref{eisensteinlemma1} and \ref{eisensteinlemma2} exactly the same way as in the proof of Theorem \ref{theorempart1}.
\end{proof}

\section{Averaging results}\label{averagingresults}

Let $\Gamma$ be a cofinite Fuchsian group. We now apply the large sieve results of \cite{cham1} for the Riemann surfaces $\Gamma \backslash \mathbb{H}$ to obtain averaging results for $E(\mathcal{H}, X; z)$. To be precise, let $a_j$ be a sequence of complex numbers and, for each cusp $\mathfrak{a}$, let $a_{\mathfrak{a}}(t)$ be a continuous function of $t$. We have the following results:
\begin{theorem} [Chamizo, \cite{cham1}]\label{lemmacham1} Given $z \in \GmodH$, $T, X>1$ and $x_1, x_2,...x_R \in [X,2X]$, if $|x_k - x_{\ell}| > \delta > 0$ for $k \neq \ell $, then
$$ \sum_{m=1}^{R} \left| \sum_{|t_j| \leq T} a_j x_m^{it_j} u_j(z) +\sum_{\mathfrak{a}}  \frac{1}{4\pi} \int_{-T}^{T} a_{\mathfrak{a}}(t) x_m^{it} E_{\mathfrak{a}}(z, 1/2 +it ) dt \right|^2$$
$$ \ll (T^2 + XT\delta^{-1}) \|a\|_{*}^2,$$
where
$$\|a\|_{*} = \left( \sum_{|t_j|\leq T} |a_j|^2 + \sum_{\mathfrak{a}} \frac{1}{4\pi} \int_{-T}^{T} |a_{\mathfrak{a}}(t)|^2dt \right)^{1/2},$$
and the \lq{$\ll$\rq} constant depends on $\G$ and $y_{\G}(z)$.
\end{theorem}

\begin{theorem} [Chamizo, \cite{cham1}]\label{lemmacham2} Given $T>1$ and $z_1, z_2,...z_R \in \GmodH$, if $ \rho(z_k,z_{\ell}) > \delta > 0$ for $k \neq \ell$, then
$$ \sum_{m=1}^{R} \left| \sum_{|t_j| \leq T} a_j u_j(z_m) + \sum_{\mathfrak{a}}  \frac{1}{4\pi} \int_{-T}^{T} a_{\mathfrak{a}}(t) E_{\mathfrak{a}}(z_m, 1/2 +it ) dt \right|^2$$
$$ \ll (T^2 + \delta^{-2}) \|a\|_{*}^2,$$
where $\|a\|_{*}$ is defined as above and 
and the  \lq{$\ll$\rq} constant depends on $\G$ and $\max y_{\G}(z_m)$.
\end{theorem}
We will use Theorem \ref{lemmacham1} to prove the following result for the radial averaging of $E(\mathcal{H}, X;z)$.
\begin{proposition} \label{averaging1} Let $X >2$ and $X_1, X_2,..., X_R \in [X,2X]$, satisfying the condition $|X_i - X_j| > \delta$ for some $\delta>0$, when $ i \neq j$. Then we have
$$\sum_{m=1}^{R} |E(\mathcal{H}, X_m ; z)|^2 \ll  R^{1/3} X^{4/3} \log X + \delta^{-1} X^2 \log^2 X ,$$
where the  \lq{$\ll$\rq} constant depends on $\G$, $\mathcal{H}$ and $z$.
\end{proposition}
\begin{theorem} \label{integral1} If $R \delta \gg X$ and $R> X^{1/2}$, then
\begin{equation} \label{riemannsumeradius}
\frac{1}{R} \sum_{m = 1}^{R} |E(\mathcal{H}, X_m ; z)|^2 \ll X \log^2 X.
\end{equation}
Letting $R$ go to infinity, we get
\begin{equation} \label{integralradius}
\frac{1}{X} \int_{X}^{2X} |E(\mathcal{H}, x  ; z)|^2 dx  \ll  X \log ^2 X.
\end{equation}
\end{theorem}
For the spatial average, we use Theorem \ref{lemmacham2} to prove the following.
\begin{proposition} \label{averaging2} Let $X >2$ and $z_1, z_2,..., z_R$ be points in $\Gamma \backslash \mathbb{H}$ away from the cusps, satisfying the condition $\rho(z_i, z_j) > \delta$ for some $\delta>0$, when $ i \neq j$. Then, we have
\begin{equation} \label{spatialsecond}
 \sum_{m=1}^{R} |E(\mathcal{H}, X ; z_m)|^2 \ll \delta^{-2} X + R^{1/3} X^{4/3} \log^2 X,
\end{equation}
and
\begin{equation} \label{spatialfourth} \sum_{m=1}^{R} |E(\mathcal{H}, X ; z_m)|^4 \ll \delta^{-2} X^2 \log^{4} X + R^{1/3} X^{8/3} \log^{3} X,
\end{equation}
where the  \lq{$\ll$\rq} constants depend on $\G$, $\mathcal{H}$ and $z$.
\end{proposition}
\begin{theorem} \label{integral2} If $R \delta^2 \gg 1$ and $R> X^{1/2}$, then, for $n=1,2$
$$\frac{1}{R} \sum_{m= 1}^{R} |E(\mathcal{H}, X ; z_m)|^{2n} \ll X^n \log^{2n} X,$$
Letting $R$ go to infinity, if $\G$ is cocompact, we get
$$ \int_{\GmodH} |E(\mathcal{H}, X ; z)|^{2n} d\mu(z)  \ll  X^{n} \log ^{2n} X.$$
\end{theorem}
Before giving the proof of the above results, we need to fix the following notation. For a function $f \in C_0^*[1, \infty)$, denote by  $E_{f}(\mathcal{H}, X; z)$ the difference 
$$E_{f}(\mathcal{H}, X; z) = A(f)(z) - \sum_{1/2 \leq s_j \leq 1} 2 d(f, t_j) \hat{u}_j u_j(z).$$
In the proofs of Theorems \ref{theorempart1} in section \ref{cocompactcase} and \ref{theorempart2} in section \ref{cofinitecase} we proved that for $\G$ cocompact or cofinite we have
\begin{equation}
E_{f^+}(\mathcal{H}, X; z)  = O(XY^{-1/2} + X^{1/2} ),
\end{equation}
\begin{equation}
E_{f^-}(\mathcal{H}, X; z)  = O(XY^{-1/2} + X^{1/2} ),
\end{equation}
\begin{equation}
 E_{f^-}(\mathcal{H}, X; z) <   E(\mathcal{H}, X; z) + O(Y + X^{1/2}\log X) <  E_{f^+}(\mathcal{H}, X; z).
\end{equation}
We begin with the proof of Proposition \ref{averaging1}.
\begin{proof} (of Proposition  \ref{averaging1}) We choose $Y$ such that $X^{1/2} \log X \ll Y \ll X$. We get
 $$E_{f^-}(\mathcal{H}, X; z) <   E(\mathcal{H}, X; z) + O(Y) <  E_{f^+}(\mathcal{H}, X; z).$$
We choose $f$ to be $f^+$ or $f^-$ as in (\ref{fplus}) and (\ref{fminus}) with $X = X_m$ and $U$ given by (\ref{uversusx}). We have
$$\sum_{m=1}^{R} |E(\mathcal{H}, X_m ; z)|^2 \ll \sum_{m=1}^{R} |E_f (\mathcal{H}, X_m ; z)|^2 + RY^2.$$
The estimates below are true for $f=f^+$ or $f^-$. We write 
\begin{eqnarray*}
S(X, z, T)  &=& 2 \sum_{T < |t_j| \leq 2T} d(f, t_j) \hat{u}_j u_j(z) \\
&&+ \frac{1}{\pi} \sum_{\mathfrak{a}} \left( \int_{T}^{2T} + \int_{-2T}^{-T} \right)  d(f, t)  \hat{ E}_{\mathfrak{a}} (1/2 + it )  E_{\mathfrak{a}} (z, 1/2 + it ) dt.
\end{eqnarray*}
We now break the set of $t_j, t$'s in the following sets
\begin{eqnarray*}
A_1 &=& \{t_j:  0< |t_j| \leq 1\}, \\
B_1 &=& \{t : 0< |t| \leq 1\}, \\
A_2 &=& \{t_j: 1 < |t_j|, \leq X^2 Y^{-2} \},\\
B_2 &=& \{t : 1 < |t| \leq X^2 Y^{-2} \},\\
A_3 &=& \{t_j :  |t_j| >  X^2Y^{-2} \},\\
B_3 &=& \{t :  |t| >  X^2Y^{-2} \}.\\
\end{eqnarray*}
Using the notation
$$ S_i(z) := 2 \sum_{t_j \in A_i}d(f, t_j) \hat{u}_j u_j(z) + \frac{1}{\pi} \sum_{\mathfrak{a}} \int_{B_i}d(f, t)  \hat{ E}_{\mathfrak{a}} (1/2+ it )  E_{\mathfrak{a}} (z, 1/2 + it ) dt, $$
$ E_f(\mathcal{H}, X ; z)$ can be written as
\begin{eqnarray*}
 E_f(\mathcal{H}, X ; z) = S_{1}(z) + S_{2}(z) + S_{3}(z).
\end{eqnarray*}
We first estimate $S_3(z)$. Using the estimates for $t_j \in \mathbb{R}$, we get the bound
\begin{eqnarray*}
\sum_{t_j \in A_3} 2 d(f, t_j) \hat{u}_j u_j(z) &\ll& \sum_{|t_j| > X^2Y^{-2} }|t_j|^{-2} \min\{t_j, X/Y \} X^{1/2} \hat{u}_j u_j(z) \\
&\ll& \sum_{t_j > X^2Y^{-2} }t_j^{-2} X^{3/2} Y^{-1} \hat{u}_j u_j(z).
\end{eqnarray*}
Using dyadic decomposition, this is bounded by
\begin{eqnarray*}
X^{3/2}Y^{-1} \sum_{n=0}^{\infty} \left(\sum_{2^n X^2 Y^{-2} < t_j \leq 2^{n+1} X^2 Y^{-2}} t_j^{-2} \hat{u}_j u_j(z) \right)
\end{eqnarray*}
and hence trivially bounded by
\begin{eqnarray*}
&\ll&X^{3/2}Y^{-1} \sum_{n=0}^{\infty} 2^{-2n} X^{-4} Y^4 \left( \sum_{2^n X^2 Y^{-2} < t_j \leq 2^{n+1} X^2 Y^{-2}}  \hat{u}_j u_j(z) \right) .
\end{eqnarray*}
Using Cauchy-Schwarz, Proposition $[7.2]$ of \cite{iwaniec} and Lemma \ref{lemmahuber2} we get the bound
\begin{eqnarray*}
&\ll&X^{-5/2} Y^{3}\sum_{n=0}^{\infty}2^{-2n} \left( \sum_{ t_j \leq 2^{n+1} X^2 Y^{-2}} |\hat{u}_j|^2 \right)^{1/2}
\left( \sum_{ t_j \leq 2^{n+1} X^2 Y^{-2}} |u_j(z)|^2 \right)^{1/2} \\
&\ll&X^{-5/2} Y^{3}  \sum_{n=0}^{\infty} 2^{-2n} (2^{n/2} XY^{-1}) (2^n X^2 Y^{-2}) \ll X^{1/2} \ll Y.
\end{eqnarray*}
Similarly we deal with the case of the Eisenstein series over $B_3$. We conclude $S_3(z) = O(Y).$ 

 We now consider $S_1(z)$. We get
$$\sum_{t_j \in A_1} d(f, t_j) \hat{u}_j u_j(z) \ll X^{1/2} \sum_{|t_j|< 1} t_j^{-2} \min\{t_j, X/Y\} \hat{u}_j u_j(z) \ll X^{1/2} \ll Y,$$
since there exist finitely many eigenvalues with spectral parameter $|t_j| \leq 1$. Similarly, we prove the $O(Y)$ bound for the Eisenstein series' contribution over $B_1$. We conclude that $ S_{1}(z) = O(Y).$
Combining all the above we get
\begin{eqnarray*}
E_f(\mathcal{H}, X; z) &=& 2 \sum_{t_j \in A_2}d(f, t_j) \hat{u}_j u_j(z) \\
&&+ \frac{1}{\pi} \sum_{\mathfrak{a}} \int_{A_2}d(f, t)  \hat{ E}_{\mathfrak{a}} ( 1/2+ it )  E_{\mathfrak{a}} (z, 1/2 + it ) dt + O(Y).
\end{eqnarray*}
Adding for $T=2^k$, $ k = 0, 1, \ldots, [\log_2 (X^2 Y^{-2})]$, we get the bound
$$E_f(\mathcal{H}, X; z) \ll \sum_{1\leq T < X^2 Y^{-2}} S(X, z, T) + O(Y),$$
and, adding for $X_1, ... X_R$, we get
\begin{equation}\label{bound}
\sum_{m=1}^{R} \left| E_f(\mathcal{H}, X_m; z) \right|^2 \ll \sum_{m=1}^{R} \left| \sum_{1\leq T < X^2 Y^{-2}} S(X_m , z, T) \right|^2 + RY^2.
\end{equation}
Cauchy-Schwarz inequality now yields
\begin{equation}\label{bound2}
\left| \sum_{ 1\leq T < X^2 Y^{-2}} S( X_m , z, T) \right|^2 \ll \log X  \sum_{1\leq T < X^2 Y^{-2}} |S ( X_m , z, T)|^2 
\end{equation}
which, combinded with the bound (\ref{bound}) gives 
\begin{equation}\label{bound3}
\sum_{m=1}^{R} \left| E_f(\mathcal{H}, X_m; z) \right|^2 \ll \log X \sum_{1\leq T < X^2 Y^{-2}}\left( \sum_{m=1}^{R} |S( X_m, z, T) |^2 \right)+ RY^2.
\end{equation}
Using the estimates of Proposition \ref{estimates1} we can now write
$$d(f,t) = X^{1/2} (a(t,Y/X) X^{it} + b(t,Y/X)  X^{-it})$$
where $a(t, Y/X)$ and $b(t, Y/X)$ are functions satisfying
$$a(t, Y/X), b(t,Y/X) \ll |t|^{-2} \min \left\{|t|, XY^{-1} \right\}.$$
We apply Theorem \ref{lemmacham1}, which implies that, for $a_j = d(f,t_j)$ and $a(t) = d(f,t)$
$$\sum_{m=1}^{R} \left|\sum_{T< |t_j| \leq 2T} d(f, t_j) \hat{u}_j u_j(z) + \frac{1}{\pi} \sum_{\mathfrak{a}} \int_{T}^{2T} d(f,t) \hat{E}_{\mathfrak{a}}(1/2 + it) E_{\mathfrak{a}}(z, 1/2 +it ) dt\right|^2 $$
is bounded by
$$ (T^2 + XT\delta^{-1}) \|a\|_{*}^2,$$
i.e.
\begin{eqnarray*}
\sum_{m=1}^{R} |S( X_m,z, T)|^2 \ll  (T^2 + XT\delta^{-1}) \|a\|_{*}^2,
\end{eqnarray*}
where
\begin{eqnarray*}
\|a\|_{*}^2 &\ll& 
\sum_{T < |t_j| \leq 2T} \Big||t_j|^{-2} \min\{|t_j|, X Y^{-1} \}X^{1/2} \hat{u}_j \Big|^2 \\
 &&+ \frac{1}{\pi} \sum_{\mathfrak{a}} \int_{T}^{2T} \left||t|^{-2} \min\{|t|, X Y^{-1} \}X^{1/2} \hat{E}_{\mathfrak{a}}(1/2+it) \right|^2.
\end{eqnarray*}
The last expression can be bounded by
\begin{eqnarray*}
 X  T^{-4} \min\{T^2, X^2 Y^{-2}  \} \left( \displaystyle \sum_{T \leq |t_j| \leq 2T} |\hat{u_j}|^2 +
\sum_{\mathfrak{a}} \int_{T}^{2T} |\hat{E}_{\mathfrak{a}}(1/2+it)|^2 dt \right) 
\end{eqnarray*}
and, using Lemma \ref{lemmahuber2}, we obtain
\begin{equation} \label{boundofnorma}
\|a\|_{*}^2 \ll XT^{-3} \min\{T^2, X^2 Y^{-2}  \}.
\end{equation}
Thus, we conclude 
$$\sum_{m=1}^{R} |S(X_m , z, T)|^2 \ll (T^2 + XT\delta^{-1} ) (XT^{-3} \min \{T^2, X^2Y^{-2} \}),$$
hence
\begin{eqnarray*}
\sum_{m=1}^{R} |E(\mathcal{H}, X_m; z)|^2 &\ll& \log X \sum_{1\leq T < X^2Y^{-2}} \left(\sum_{m=1}^{R}|S(X_m, z, T)|^2\right) + RY^2 \\
&\ll& \log X \sum_{1\leq T < X^2Y^{-2}} (T^2 + XT\delta^{-1} ) (XT^{-3} \min \{T^2, X^2 Y^{-2} \}) 
+ RY^2.
\end{eqnarray*}
We get the bound 
\begin{eqnarray*}
\sum_{m=1}^{R} |E(\mathcal{H}, X_m; z)|^2 &\ll& X \log X \left( \sum_{1\leq T < XY^{-1}} T \right) +  X^2\delta^{-1}  \log X \left(\sum_{1\leq T < XY^{-1}} 1 \right)  \\
&&+ X^3Y^{-2} \log X \left(\sum_{X Y^{-1}\leq T < X^2 Y^{-2}}T^{-1} \right) \\
&&+ X^4 \delta^{-1}Y^{-2}\log X  \left( \sum_{X Y^{-1}\leq T < X^2Y^{-2}} T^{-2} \right) +RY^2.
\end{eqnarray*}
Trivial bounds for each term seperately yield the bound
\begin{eqnarray*}
\sum_{m=1}^{R} |E(\mathcal{H}, X_m; z)|^2 \ll  X^2 Y^{-1} \log  X+ \delta^{-1} X^2 \log ^2 X + RY^2.
\end{eqnarray*}
The optimal choice for $Y$ is $Y = R^{-1/3} X^{2/3}$ which implies the bound
$$\sum_{m=1}^{R} |E(\mathcal{H}, X_m; z)|^2 \ll  R^{1/3} X^{4/3} \log X + \delta^{-1} X^2 \log ^2 X.$$
\end{proof}

\begin{proof} (of Theorem \ref{integral1})
Choosing $\delta^{-1} \ll R X^{-1}$ and $R > X^{1/2}$ in the bound
$$\sum_{m=1}^{R} |E(\mathcal{H}, X_m; z)|^2 \ll  R^{1/3} X^{4/3} \log  X + \delta^{-1} X^2 \log ^2 X,$$
we get
$$\sum_{m=1}^{R} |E(\mathcal{H}, X_m; z)|^2 \ll  R^{1/3} X^{4/3} \log  X + R X \log ^2 X \ll R X \log^2 X$$
and we get the bound (\ref{riemannsumeradius}). For the bound (\ref{integralradius}), we take the points $X_i$ equally spaced in the interval $[X,2X]$ with $\delta= R^{-1} X$. As $R \to \infty$,
$$\sum_{m=1}^{R} |E(\mathcal{H}, X_m; z)|^2 \frac{X}{R} \to \int_{X}^{2X} |E(\mathcal{H}, x ; z)|^2 dx,$$
hence
$$\frac{1}{X} \int_{X}^{2X} |E(\mathcal{H}, x ; z)|^2 dx \ll X \log ^2 X.$$

\end{proof}

\begin{proof} (of Proposition \ref{averaging2}) For a sequence $\{a_k\}$, Cauchy-Schwarz inequality implies 
\begin{eqnarray*}
\left( \sum_{k=0}^{n} a_k \right)^2 \ll \sum_{k=0}^{n} (n+1-k)^2 a_k^2, \quad \left( \sum_{k=0}^{n} a_k \right)^2 \ll \sum_{k=0}^{n} (k+1)^2 a_k^2.
\end{eqnarray*}
The first inequality for $a_k = S( X , z_m, 2^k) $ implies the bound
\begin{equation}\label{bound4}
\left| \sum_{ 1\leq T < X^2 Y^{-2}} S(X , z_m, T) \right|^2 \ll \sum_{1\leq T < X^2 Y^{-2}}\left| \log T^{-1} X^2 Y^{-2}+1 \right|^2 |S( X, z_m, T)|^2,
\end{equation}
whereas the second gives
\begin{equation}\label{bound7}
\left| \sum_{ 1\leq T < X^2 Y^{-2}} S( X , z_m, T) \right|^2 \ll \sum_{1\leq T < X^2 Y^{-2}} (\log T +1 )^2 |S( X , z_m, T)|^2.
\end{equation}
The bounds (\ref{bound4}) and (\ref{bound7}) give
\begin{equation}\label{bound8}
\left| \sum_{ 1\leq T < X^2 Y^{-2}} S(X , z_m, T) \right|^2 \ll \sum_{1\leq T < X^2 Y^{-2}}| c_T|^2 |S( X , z_m, T)|^2,
\end{equation}
where $c_T = \min \left \{ \log T^{-1} X^2 Y^{-2} +1, \log T + 1 \right\}$. 
Using Theorem \ref{lemmacham2}, bound (\ref{boundofnorma}) and summing over $T=2^k$, $ k = 0, 1, \ldots, [\log_2 (X^2 Y^{-2})]$, we get
\begin{eqnarray*}
\sum_{m=1}^{R} |E(\mathcal{H}, X; z_m)|^2 &\ll& \sum_{1 \leq T < X^2 Y^{-2} } |c_T|^2  (T^2 + \delta^{-2}) (XT^{-1} \min \{1, X^2 T^{-2} Y^{-2} \}) + RY^2 \\
&\ll& \delta^{-2} X + X^2 Y^{-1} \log^2 X + R Y^2,
\end{eqnarray*}
where the last bound yields as in the proof of Proposition \ref{averaging1}. The bound (\ref{spatialsecond}) is obtained for $Y = X^{2/3} R^{-1/3}$. For the fourth moment, we use H\"older's inequality to prove
\begin{equation}\label{bound9}
\sum_{m=1}^{R} \left| E_f(\mathcal{H}, X; z_m) \right|^4 \ll \log^3 X \sum_{1\leq T < X^2 Y^{-2}}\left( \sum_{m=1}^{R} |S(X, z_m, T) |^4 \right)+ RY^4,
\end{equation}
We can now finish the proof assuming the following large sieve inequality
\begin{equation}\label{bound10}
\sum_{m=1}^{R} \left| \sum_{|t_j| \leq T} a_j u_j(z_m) + \frac{1}{4 \pi} \sum_{\mathfrak{a}} \int_{-T}^{T} a_{\mathfrak{a}}(t) E_{\mathfrak{a}}(z_m, 1/2+it) dt \right|^4 \ll (T^4 +T^2 \delta^{-2}) \|a\|_{*}^4.
\end{equation}
For the proof see \cite{cham0}. We can now derive the second part of the proposition applying (\ref{bound10}) to the $a_j = d(f,t_j)$ and $a_{\frak{a}}(t) = d(f,t)$. 
\end{proof}

\begin{proof} (of Theorem \ref{integral2})
Consider first the $n=1$ case. Choosing $\delta^{-2} \ll R $ and $R > X^{1/2}$ in the bound
$$\sum_{m=1}^{R} |E(\mathcal{H}, X; z_{m})|^2 \ll  \delta^{-2} X + R^{1/3} X^{4/3} \log ^2 X,$$
we get
$$\frac{1}{R} \sum_{m=1}^{R} |E(\mathcal{H}, X; z_{m})|^2 \ll X + R^{-2/3} X^{4/3} \log ^2 X     \ll X \log^2 X,$$
and the first part follows. For the integral estimate we notice that as $R \to \infty$
$$\frac{1}{R} \sum_{m=1}^{R} |E(\mathcal{H}, X; z_m)|^2 \to \int_{\GmodH} |E(\mathcal{H}, X ; z)|^2 d \mu(z).$$
Hence
$$\int_{\GmodH} |E(\mathcal{H}, X ; z)|^2 d \mu(z)  \ll X \log^2  X.$$
The $n=2$ case follows in exactly the same way.
\end{proof}

For the error term $E(X;z,w)$ of the classical hyperbolic lattice point problem, the optimal bound is conjectured to be 
$$E(X;z,w) =O(X^{1/2 + \epsilon})$$
for every $\epsilon > 0$. This is supported by the $\Omega$-results of Phillips-Rudnick \cite{phirud} and the averaging results in \cite{cham2}.

 Theorems \ref{integral1} and \ref{integral2} lead us to formulate the analogous conjecture.
\begin{conjecture} For $\G$ cocompact or cofinite and $\mathcal{H}$ a hyperbolic conjugacy class of $\G$, the error term $E(\mathcal{H}, X;z)$ satisfies the bound
$$E(\mathcal{H}, X;z) = O(X^{1/2 + \epsilon})$$
for every $\epsilon > 0$.
\end{conjecture}

\section{Arithmetic applications}\label{arithmeticapplications}

 In this section we are interested in arithmetic corollaries of our results. We use the geometric interpretation to get an arithmetic interpretation of the quantity $N(\mathcal{H}, X;z)$. We restrict our attention to $\G = \slz$.

 Fix a point $z$ in $\mathbb{H}$. Huber's interpretation in \cite{huber} shows that $N(\mathcal{H}, X;z)$ counts $\gamma$ in $\Gamma/  \langle g \rangle $ such that $\cos v \geq X^{-1}$, where $v$ is the angle defined by the ray from $0$ to $\gamma z$ and the geodesic $\{yi, y>0\}$. Denote by $c$ the geodesic from $\gamma z$ perpendicular to $\{yi, y>0\}$ and by $\ell(c)$ its length. Then
\begin{equation}\label{arithm1}
 \rho(\gamma z,\{ iy, y>0 \})=  \ell(c) = \int_{\pi/2 - v}^{\pi/2} \csc (t) \, dt = \log \left(\frac{1 + \sin v}{\cos v} \right).
\end{equation}
 On the other hand, the distance of $\gamma z$ to the imaginary axis is given by
\begin{equation} \label{arithm2}
\cosh \rho(\gamma z, i|\gamma z|) = \frac{|\gamma z|}{\Im(\gamma z)}.
\end{equation}
For
$$\gamma = \left( \begin{array}{cc} A & B \\ C & D \end{array} \right),$$
we define
\begin{equation} \label{arithm3}
f_z (A,B,C,D) =\cosh \rho(\gamma z, i|\gamma z|) = \frac{|\gamma z|}{\Im(\gamma z)} = \frac{|Az+B||Cz+D|}{\Im (z)}.
\end{equation}
Using $|Cz+D| = |C\bar{z} +D|$ and $AD-BC=1$, $f_z$ can be written in the form 
\begin{equation} \label{arithm4}
 f_z(A,B,C,D) =  \frac{ \left( \left(AC|z|^2+ BD + AD \Re(z) + BC \Re(z) \right)^2 + \Im(z)^2 \right)^{1/2}}{\Im(z)}.
\end{equation}
Using (\ref{arithm1}), the condition $\cos v \geq X^{-1}$ can now be written as
\begin{equation}\label{arithm5}
 f_z(A,B,C,D) \leq  \cosh \left(\log (X + \sqrt{X^2-1})\right)= X.
\end{equation}
Let $z= \a+ \b i$. Inequality (\ref{arithm5}) takes the form
\begin{equation}\label{arithm6}
\left|(\a^2+\b^2) AC + BD + \a AD + \a BC\right| \leq \b \sqrt{X^2-1} =\b X + O(X^{-1}). 
\end{equation}
To get simple results, we take specific choices for $z$.

\subsection{Quadratic forms:} For the basics of indefinite quadratic forms we refer to \cite{sarnak}. Let $Q(x,y) = ax^2+bxy +cy^2$ be a primitive indefinite quadratic form in two variables, i.e. $(a,b,c) = 1$ and $b^2-4ac = d>0$ is not a square. We denote $Q$ by $[a,b,c]$. Two forms $[a,b,c]$, $[a',b',c']$ are called equivalent ($[a,b,c] \sim [a',b',c']$) if there is a $\gamma \in \slz$ such that 
$$
 \left( \begin{array}{cc} a' & b'/2 \\ b'/2 & c' \end{array} \right) =  \gamma^{t}  \left( \begin{array}{cc} a & b/2 \\ b/2 & c \end{array} \right) \gamma.
$$
The automorphs of $Q$ is the group $\hbox{Aut}(Q) = \Gamma \subset \slz$ which fixes $Q$, under the action above. This group is infinite and cyclic, with generator
$$M_{[a,b,c]} =   \left( \begin{array}{cc} \frac{t_0 - b u_0}{2} & -cu_0 \\ au_0 & \frac{t_0 + bu_0}{2} \end{array} \right),$$
where $t_0, u_0 > 0$ is the fundamental solution of Pell's equation $x^2-dy^2 =4$. Since $t_0 > 2$, the matrix $M_{[a,b,c]}$ is hyperbolic. We denote by $\varepsilon_d$ the quantity
$$ \varepsilon_d = \frac{t_0+\sqrt{d} u_0}{2}.$$
The quadratic form $Q$ is associated with two real quadratic numbers
$$\theta_1 = \frac{-b + \sqrt{d}}{2a}, \quad \theta_2 = \frac{-b - \sqrt{d}}{2a},$$
the roots of the polynomial $a \theta^2 + b \theta + c$.

 The main reason we are interested in quadratic forms is that, according to the next proposition, they are in one-to-one correspondence with hyperbolic conjugacy classes of the modular group (a proof of this result can be found in \cite[p.~232]{sarnak}).
\begin{proposition}\label{prop} Define the map $\phi$ by
$$\phi([a,b,c]) = M_{[a,b,c]}.$$
Then
\\ a) $\phi$ is a bijective map of the set of primitive indefinite quadratic forms onto the set of primitive hyperbolic elements of $\slz$.
\\ b) $\phi$ commutes with the action of $\slz$: $[a,b,c] \sim [a',b',c']$ iff $M_{[a,b,c]}$ is conjugate to $M_{[a',b',c']}$.
\end{proposition}
The matrix $M_{[a,b,c]}$ has eigenvalues:
\begin{equation} \label{eigenvalues}
\lambda_{1,2} = \frac{t_0 \pm \sqrt{d} u_0}{2} = \varepsilon_d^{\pm 1}.
\end{equation}
Its diagonalization is
\begin{equation} \label{tmatrix}
M_{[a,b,c]} = T  \left( \begin{array}{cc} \lambda_1 & 0 \\ 0 & \lambda_2 \end{array} \right) T^{-1}, \quad
T =  \left( \begin{array}{cc} \theta_1 & \theta_2 \\ 1 & 1 \end{array} \right).
\end{equation}
Calculations imply that
$$M_{[a,b,c]}^n = \frac{1}{\theta_1 - \theta_2}  \left( \begin{array}{cc} \lambda_1^n \theta_1 - \lambda_2^n \theta_2 & (\lambda_2^n - \lambda_1^n) \theta_1 \theta_2 \\ \lambda_1^n - \lambda_2^n & \lambda_2^n \theta_1 - \lambda_1^n \theta_2 \end{array} \right) \in \slz.$$ 

 Let us first examine a simple case. Consider the case $b=0$, i.e. we consider the form $Q(x,y)= ax^2+cy^2$. Set $\theta = \theta_1 = - \theta_2= \sqrt{d}/2a $. Thus, $\theta^2 = -c/a >0$. In this case, $M_{[a,b,c]}^n$ takes the form
\begin{equation} \label{matrixmn}
M_{[a,0,c]}^n = M^n = \frac{1}{2 \theta}  \left( \begin{array}{cc} (\lambda_1^n + \lambda_2^n) \theta & (\lambda_1^n - 
\lambda_2^n) \theta^2 \\ \lambda_1^n - \lambda_2^n & (\lambda_1^n + \lambda_2^n) \theta \end{array} \right) \in \slz.
\end{equation}
We now prove an application of Theorems \ref{theorem1} and \ref{average1}.
\begin{proposition} \label{arithmetic1} Given $Q(x,y) = ax^2+cy^2$ with $\theta^2 = -c/a>0$ and $-ac$ not a square, let $F_Q$ denote the indefinite quadratic form 
$$F_Q (\a, \b, \g , \delta) = \a^2 -\frac{a}{c} \b^2  +\frac{c}{a} \g^2  -  \delta^2.$$
Let $P(X)$ be the number  of solutions $(\a, \b, \g, \delta) \in \Z^4$ such that $\a \delta - \b \g=1$ and
$$\left| F_Q (\a, \b, \g, \delta) \right| \leq X,$$
under the equivalence: $(\a, \b, \g, \delta) \sim (\a', \b', \g', \delta')$ iff there exists an integer $n$ such that
\begin{eqnarray*} \left( \begin{array}{cc} \a & \b \\ \g & \delta \end{array} \right)  = M^n  \left( \begin{array}{cc} \a' & \b' \\ \g' & \delta' \end{array} \right).
\end{eqnarray*}
Here $M^n$ is given by (\ref{matrixmn}). Then
\\a) $P(X)$ satisfies
\begin{eqnarray*}
P(X) = \frac{6 \log \varepsilon_d}{\pi} X + E(X), 
\end{eqnarray*}
with $\epsilon_d$ given by (\ref{eigenvalues}) and $E(X) = O( X^{2/3})$.
\\b) $E(X)$ satisfies the average bound 
$$\frac{1}{X} \int_{X}^{2X} |E(x)|^2 dx \ll  X \log^2 X,$$
where the \lq{$\ll$\rq}  constant depends on the quadratic form $Q$.
\end{proposition}
\begin{proof}
 For a), let $\ell$ be the invariant closed geodesic of $M$. Conjugating $\G$ with $T$, we bring $\ell$ on the imaginary axis. Let also $z = T^{-1} (i)$. In this case $N(\mathcal{H}, X;z)$ counts 
\begin{eqnarray*}
\g =  \left( \begin{array}{cc} A & B \\ C & D \end{array} \right) \in \slz / \langle M \rangle \simeq T^{-1} \slz T /  G,
\end{eqnarray*}
where
\begin{eqnarray*}
G = \left\langle  \left( \begin{array}{cc} \lambda_1 & 0 \\ 0 & \lambda_2 \end{array} \right) \right\rangle,
\end{eqnarray*}
such that asymptotically
\begin{eqnarray*}
|AC + BD| \leq X.
\end{eqnarray*}
For
\begin{eqnarray*}
\g = \left( \begin{array}{cc} A & B \\ C & D \end{array} \right) = T^{-1}  \left( \begin{array}{cc} \a & \b \\ \g & \delta \end{array} \right)   T \in T^{-1} \slz T 
\end{eqnarray*}
we get
\begin{eqnarray*}
2 |AC + BD|= |F_Q(\a, \b, \g, \delta)|.
\end{eqnarray*}
Thus, $ N(\mathcal{H}, X;z)$ counts points $\slz$ under the extra equivalence that comes from the quotient with $\langle M \rangle$ such that
\begin{eqnarray*}
| F_Q(\a, \b, \g, \delta)| \leq 2 X,
\end{eqnarray*}
i.e.  $ P(X) = N(\mathcal{H}, X/2;z).$ 
By Theorem \ref{theorempart2} and the fact that $\slz$ has no eigenvalues $\lambda  \in (0, 1/4)$, we get 
\begin{eqnarray*}
 P(X) = \hat{u}_0 u_0(z) X + O(X^{2/3}).
\end{eqnarray*}
For $\G = \slz$ we have 
\begin{eqnarray*}
u_0(z) = \sqrt{\frac{3}{\pi}}, \quad \hat{u}_0 =  \sqrt{\frac{3}{\pi}} \frac{\mu}{\nu}.
\end{eqnarray*}
Since $M$ is primitive $\nu=1$ and, for $\slz$, we know that $\mu = 2 \log \varepsilon_d$, which equals the length of the closed geodesic $\ell$ (see for example Corollary $1.5$ of \cite{sarnak}). Part a) now follows. Part b) follows immediately as $E(X) = E(\mathcal{H}, X/2;z)$.
\end{proof}

\begin{remark} Note that, for $z=i$, by (\ref{arithm3}) we count solutions of 
$$(A^2+B^2)(C^2+D^2) \leq X^2,$$ 
with restrictions. The more general case $b \neq 0$ or $z \neq i$ leads to a more complicated quadratic form (see (\ref{arithm6}) and (\ref{tmatrix})). 
\end{remark}

\begin{remark} The arithmetic corollaries of the classical hyperbolic lattice point problem (see \cite{cham2}) differ from ours in the fact that the quadratic forms in \cite{cham2} are positive definite, e.g. for $z=w=i$ one gets $4u(\gamma i, i)+2 = \a^2 + \b^2 + \g^2 + \delta^2$. The quadratic form $F_Q (\a, \b, \g, \delta)$ is indefinite.
\end{remark}

\subsection{Hecke operators:}
Applying Hecke operators as in \cite{iwaniec,cham2} for the classical lattice point counting problem, we can count solutions of  $|F(\a,\b, \g, \delta)| \leq X$ lying in the hypersurface $\a \delta - \b \g = n,$ with $n>1$. Let $\G_n$ be the set
\begin{eqnarray*}
\G_n = \left \{  \left( \begin{array}{cc} \a & \b \\ \g & \delta \end{array} \right) \in \Z^{2 \times 2} : \a \delta - \b \gamma=n \right \}.
\end{eqnarray*} 
For $n \in \N$, let $T_n : \mathcal{A} (\GmodH) \to \mathcal{A} (\GmodH)$ be the $n$-th Hecke operator, see  \cite[section 8.5, chapter 12]{iwaniec}, defined by
\begin{eqnarray*}
T_n (f) (z)  = \frac{1}{\sqrt{n}} \sum_{\tau \in \G \backslash \G_n} f(\tau z).
\end{eqnarray*}
As the Hecke operators commute with $\Delta$, we choose a joint orthonormal basis $u_j$. We denote by $\lambda_j(n)$ the eigenvalue of $T_n$ for $u_j(z)$, i.e.
\begin{eqnarray*}
T_n u_j(z) = \lambda_j(n) u_j(z),
\end{eqnarray*}
and $\eta_{t}(n)$ for the Eisenstein series, i.e.
\begin{eqnarray*}
T_n E_{\infty}(z, 1/2+it) =  \eta_t (n) E_{\infty}(z, 1/2+it),
\end{eqnarray*}
where 
\begin{eqnarray*}
 \eta_t (n) = \sum_{ad=n} \left( \frac{a}{d} \right)^{it}.
\end{eqnarray*}
Notice that counting solutions
\begin{eqnarray*}
\left| F_{Q} (\a, \b, \g, \delta) \right| \leq X
\end{eqnarray*}
with $\a \delta - \b \g =n$ is equivalent to counting solutions
\begin{eqnarray*}
\left| f_i (A,B,C,D) \right| \leq  n X
\end{eqnarray*}
with $AD -BC =n$.

 We apply $T_n$ on both expressions of $A(f) (z)$. Applying $T_n$ to the spectral expansion (\ref{cofinitecaseexpansion}) we get
\begin{equation} \label{spectralexp}
T_n A(f)(z) =  \sum_{j} c(f, t_j) \lambda_j(n) u_j(z) + \frac{1}{4 \pi} \int_{-\infty}^{\infty} c_{\infty}(f,t)  \eta_t (n)  E_{\infty} (z, 1/2 + it ) dt.
\end{equation}
On the geometric side, we have
\begin{eqnarray*} 
T_n A(f)(z) =  \frac{1}{\sqrt{n}} \sum_{\tau \in \G \backslash \G_n} \left( \sum_{\gamma \in \mathcal{H}} f \left( \frac{\cosh \rho( \tau^{-1} \gamma \tau z, z) -1}{\cosh \mu(\g)-1} \right)\right).
\end{eqnarray*}
If $\mathcal{H}$ is the conjugacy class of the primitive hyperbolic matrix $M$, we define the set 
\begin{eqnarray*}
\mathcal{H}_n = \{ \g^{-1} M \g : \g \in \G_n \}.
\end{eqnarray*}
The set $\mathcal{H}_n$ is in one-to-one correspondence with the quotient set $\G_n / \langle M \rangle$. Notice also that $\mu( \tau \g \tau^{-1}) = \mu(\g)$. Therefore,
\begin{equation} \label{geometricexp}
T_n A(f)(z) = \frac{1}{\sqrt{n}} \sum_{\gamma \in \mathcal{H}_n} f \left( \frac{\cosh \rho( \gamma z, z) -1}{\cosh \mu(\g)-1} \right).
\end{equation}
Using that $|\lambda_j(n) | \leq \lambda_0(n) = \sigma(n) n^{-1/2}$ and the uniform bound $ |\eta_t (n)| \ll d(n) \ll \lambda_0(n)$ we conclude the following.
\begin{proposition}\label{arithmetic3}
Denote with $P_{Q,n}(X)$ the number  of solutions $(\a, \b, \g, \delta) \in \Z$ such that $\a \delta - \b \g=n$ and
$$\left| F_{Q} (\a, \b, \g, \delta) \right| \leq X,$$
under the equivalence $\sim$ such that: $(\a, \b, \g, \delta) \sim (\a', \b', \g', \delta')$ iff there exists an integer $m$ such that
\begin{eqnarray*} \left( \begin{array}{cc} \a & \b \\ \g & \delta \end{array} \right)  = M_{[a,b,c]}^m  \left( \begin{array}{cc} \a' & \b' \\ \g' & \delta' \end{array} \right).
\end{eqnarray*}
Then
\\ a) $ P_{Q,n}(X)$ has the asymptotic behaviour
\begin{eqnarray*}
 P_{Q,n}(X) = \frac{6 \log \varepsilon_d}{\pi} \frac{\sigma(n)}{n} X + E_n(X),
\end{eqnarray*}
with 
\begin{eqnarray*}
E_n(X) = O \left(\frac{\sigma(n)}{n^{2/3}} X^{2/3} \right).
\end{eqnarray*}
b) $E_n(X)$ satisfies the bound
$$ \frac{1}{X} \int_{X}^{2X} | E_n (x)|^2 dx  \ll \sigma^2(n) \frac{X}{n} \log^2  \left(\frac{X}{n}\right),$$
where the  \lq{$\ll$\rq} constant depends on the quadratic form $Q$.

\end{proposition}


\begin{thebibliography}{99}


\bibitem{bateman} A. Erd\'elyi, W. Magnus, F. Oberhettinger and F. G. Tricomi.
\emph{Higher transcendental functions. Vol. I.} Based on notes left by Harry Bateman. With a preface by Mina Rees. With a foreword by E. C. Watson. Reprint of the 1953 original. Robert E. Krieger Publishing Co., Inc., Melbourne, Fla., 1981. xiii+302 pp. 

\bibitem{bruin} P. Bruin.
\emph{Explicit bounds on automorphic and canonical Green functions of Fuchsian groups.} Mathematika 60, no. 2, 257--306, 2014. ​ 

\bibitem{cham0} F. Chamizo.
\emph{Topics in Analytic Number Theory.} Doctoral Thesis, Universdad Aut\'onoma de Madrid, 1994.

\bibitem{cham1} F. Chamizo.
\emph{The large sieve in Riemann surfaces.} Acta Arith. 77, no. 4, 303--313, 1996. 

\bibitem{cham2} F. Chamizo.
\emph{Some applications of large sieve in Riemann surfaces.} Acta Arith. 77, no. 4, 315--337, 1996.

\bibitem{duke} W. Duke, Z. Rudnick and P. Sarnak.
\emph{Density of integer points on affine homogeneous varieties.} Duke Math. J. 71, no. 1, 143--179, 1993.

\bibitem{eskin} A. Eskin and C. McMullen. 
\emph{Mixing, counting, and equidistribution in Lie groups.} Duke Math. J. 71, no. 1, 181--209, 1993. 

\bibitem{fay} J. D. Fay.
\emph{Fourier coefficients of the resolvent for a Fuchsian group.} J. Reine Angew. Math. 293/294, 143--203, 1977.

\bibitem{gradry} I. S. Gradshteyn and I. M. Ryzhik. 
\emph{Table of integrals, series, and products. Translated from the Russian.} Translation edited and with a preface by Alan Jeffrey and Daniel Zwillinger. Seventh edition. Elsevier/Academic Press, Amsterdam, 2007. xlviii+1171 pp.

\bibitem{gjm} D. Garbin, J. Jorgenson and M. Munn. 
\emph{On the appearance of Eisenstein series through degeneration.} Comment. Math. Helv. 83, no. 4, 701--721, 2008.

\bibitem{good} A. Good.
\emph{Local analysis of Selberg's trace formula. Lecture Notes in Mathematics, 1040.} Springer-Verlag, Berlin, 1983. i+128 pp.

\bibitem{gunth}  P. G\"unther.
\emph{Gitterpunktprobleme in symmetrischen Riemannschen Räumen vom Rang 1.} Math. Nachr. 94, 5--27, 1980.

\bibitem{hillparn} R. Hill and L. Parnovski. 
\emph{The variance of the hyperbolic lattice point counting function.} Russ. J. Math. Phys. 12, no. 4, 472--482, 2005. 

\bibitem{hub} H. Huber.
\emph{ \"Uber eine neue Klasse automorpher Funktionen und ein Gitterpunktproblem in der hyperbolischen Ebene. I. } Comment. Math. Helv. 30, 20--62, 1956. 

\bibitem{huber} H. Huber.
\emph{Ein Gitterpunktproblem in der hyperbolischen Ebene.} J. Reine Angew. Math. 496, 15--53, 1998. 

\bibitem{iwaniec} H. Iwaniec
\emph{Spectral methods of automorphic forms.} Second edition. Graduate Studies in Mathematics, 53. American Mathematical Society, Providence, RI; Revista Matemática Iberoamericana, Madrid, 2002. xii+220 pp. 

\bibitem{martin} K. Martin, M. McKee and E. Wambach. 
\emph{A relative trace formula for a compact Riemann surface.} Int. J. Number Theory 7, no. 2, 389--429, 2011. ​

\bibitem{parkpaul} J. Parkkonen and F. Paulin.
\emph{On the hyperbolic orbital counting problem in conjugacy classes.} Math. Z. 279, no. 3-4, 1175--1196, 2015.

\bibitem{patt} S. J. Patterson.
\emph{A lattice-point problem in hyperbolic space.} Mathematika 22, no. 1, 81--88, 1975. 

\bibitem{phirud} R. Phillips and Z. Rudnick. 
\emph{The circle problem in the hyperbolic plane.} J. Funct. Anal. 121, no. 1, 78--116, 1994. 

\bibitem{sarnak} P. Sarnak.
\emph {Class numbers of indefinite binary quadratic forms.} J. Number Theory 15, no. 2, 229--247, 1982. 

\bibitem{selberg} A. Selberg.
\emph{Equidistribution in discrete groups and the spectral theory of automorphic forms}, http://publications.ias.edu/selberg/section/2491

\bibitem{tsuzuki} M. Tsuzuki. 
\emph{Spectral square means for period integrals of wave functions on real hyperbolic spaces.} J. Number Theory 129, no. 10, 2387--2438, 2009.




\end{thebibliography}
\end{document}